\documentclass{amsart}
\usepackage{amsmath}
\usepackage{mystyle}
\usepackage{tabu}
\usepackage{largearray}
\usepackage{manfnt}
\title[The $P_2^1$ Margolis homology of $\tmf$]{The $P_2^1$ Margolis homology of connective topological modular forms}
\author[Prasit Bhattacharya]{Prasit Bhattacharya}
\address{Department of Mathematics, University of Virginia, Kerchoff hall, Charlottesville, VA 22904}
\email{pb9wh@virginia.edu}
\author[Irina Bobkova]{Irina Bobkova}
\address{Department of Mathematics, Texas A\&M University, College Station, TX 77843}
\email{ibobkova@math.tamu.edu}
\author[Brian Thomas]{Brian Thomas}
\email{bt3hy@virginia.edu}

\newcommand{\FF}{\mathbb{F}}
\newcommand{\W}{\mathcal{W}}
\newcommand{\heart}{\ensuremath\heartsuit}

\newcommand{\e}[1]{\overline{#1}}

\begin{document} 
\setcounter{tocdepth}{1}
%\keywords{Steenrod algebra, Margolis homology, topological modular forms}

\begin{abstract} 
The element $\Pt$ of the mod $2$ Steenrod algebra $\Steen$ has the property $(\Pt)^2=0$. This property allows one to view $\Pt$ as a differential on $H_*(X, \FF_2)$ for any spectrum $X$. 
Homology with respect to this differential, $\MM(X, \Pt)$, is called the $\Pt$ Margolis homology of $X$. In this paper we give a complete calculation of the $\Pt$ Margolis homology of the $2$-local spectrum of topological modular forms $\tmf$ and identify its $\mathbb{F}_2$ basis via an iterated algorithm. We apply the same techniques to calculate $\Pt$ Margolis homology for any smash power of $\tmf$.
\end{abstract}

\maketitle
\tableofcontents  
\begin{conv*}Throughout this paper we work in the stable homotopy category of spectra localized at the prime $2$.  
\end{conv*}

\section{Introduction} \label{Sec:intro}
The connective $E_{\infty}$ ring \emph{spectrum of topological modular forms} $\tmf$ has played a vital role in computational aspects of chromatic homotopy theory over the last two decades \cite{MR2648680}, \cite{tmfbook}. It is essential for detecting information about the chromatic height $2$, and it has the rare quality of having rich Hurewicz image. There is a $K(2)$-local equivalence \cite{HopkinsMahowald} \[L_{K(2)}\tmf \simeq E_2^{hG_{48}} \]
 where $E_2$ is the second Morava $E$-theory at $p=2$ and  $G_{48}$ is the maximal finite subgroup of the Morava stabilizer group $\mathbb{G}_2.$  The spectrum $E_2^{hG_{48}}$ can be used to build the $K(2)$-local sphere spectrum  (see~\cite{BobkovaGoerss}).  The homotopy groups of $\tmf$ approximate both
the stable homotopy groups of spheres and the ring of integral modular forms. In many senses, $\tmf$ is the chromatic height $2$ analogue of connective real $K$-theory $\mathit{ko}$. Further, the homotopy groups of $\tmf$ are completely known \cite{tbauer}. 
 
Let us now recall the definition of the element $\Pt \in \Steen$. Milnor described the mod $2$ dual Steenrod algebra  $\Steen_*$ as the graded polynomial algebra \cite[App. 1]{Milnor_Steenrod} 
\[ \Steen_* \cong \Ft[\xi_1, \xi_2, \xi_3, \dots],\]
where $|\xi_i| = 2^{i} - 1$.   The Steenrod algebra $\Steen$ has an $\FF_2$-basis dual to the monomial basis of $\Steen_*$. The elements of the $\FF_2$-basis of $\Steen$ which are dual to $\xi_t^{2^s}$ are denoted by $\PP_t^s$, and the elements $\PP_t^0$ are denoted by $\Q_{t-1}$. 
 When $s<t$, the elements $\PP_t^s$ are exterior power generators, i.e. $(\PP_t^s)^2 = 0$. Thus, any left $\Steen$-module $K$ 
 can be regarded as a complex with differential given by the left multiplication by $\PP_t^s$ (for $s < t$). This leads to the following definition.
 \begin{defn}[\cite{Margolis_book}] \label{margolis}
 Let $K$ be any left  $\Steen$-module and $0 \leq s<t$.  Let $${^{L}\mathcal{P}}_t^s:K \to K$$ denote the left action by $\PP_t^s$. The \emph{left $\PP_t^s$ Margolis homology group} of $K$, $\MM^L(K, \PP_t^s)$, is defined as
\[ \MM^L(K, \PP_t^s) := \frac{\Kera{^{L}\mathcal{P}}_t^s: K \to K}{\Ima {^{L}\mathcal{P}}_t^s: K \to K}.\]
For a right $\Steen$-module $K$, one can similarly define the \emph{right $\PP_t^s$ Margolis homology group} of $K$ as
  \[ \MM^R(K, \PP_t^s) := \frac{\Kera{^{R}\mathcal{P}}_t^s: K \to K}{\Ima {^{R}\mathcal{P}}_t^s: K \to K}\] 
 where ${^{R}\mathcal{P}}_t^s$ is the right action by $\PP_t^s$ on $K$. 
 \end{defn}

\begin{notn}
For a spectrum $X$, $\MM(X, \PP_t^s)$ will denote  $\MM^L(H^*(X), \PP_t^s)$ or equivalently $\MM^R(H_*(X), \PP_t^s)$. 
\end{notn}

 Computations of Margolis homology underly many essential computations in homotopy theory. For example, Adams work on $BP\langle 1 \rangle$ cooperations \cite{A74} relies on the computations of $\MM(BP\langle 1 \rangle, \Q_i)$ for $i=0,1$. Calculations like $\MM(bo, \Q_i)$ for $i=0,1$ are essential ingredients in the work of Mahowald on $bo$-resolutions \cite{Mah81}.  More recently, Culver described $BP\langle 2 \rangle$ resolutions \cite{C} by understanding 
 $\MM(BP\langle 2 \rangle, \Q_i)$ for $i=0,1,2$. Computation of $\MM(\tmf^{\sma n}, \Q_2)$ is an essential ingredient in \cite{BBBCX}.
 
 The element $\Q_i$ is primitive for all $i \in \mathbb{N}$. In other words, the comultiplication map $\Delta$ on  $\Steen$ sends $\Q_i$ to 
 \begin{equation} \label{coprodQ}
 \Delta(\Q_i) = \Q_i \otimes 1 + 1\otimes \Q_i.
 \end{equation}
 Consequently, $\Q_i$  acts on $H_*(X)$ as a derivation, namely it follows the Leibniz rule
\[ \Q_i(x y) = \Q_i(x) \cdot y + x \cdot \Q_i(y),\]
whenever $X$ is a ring spectrum.  The Leibniz rule implies the K{\"u}nneth isomorphism \cite[Proposition~17, pg 343]{Margolis_book}
\[ \MM(X \otimes Y, \Q_i)  \cong \MM(X, \Q_i) \otimes  \MM(Y, \Q_i)\]
and hence,  $\MM(X, \Q_i)$ is an $\FF_2$ algebra whenever $X$ is a ring spectrum. As a result, computation of $\Q_i$ Margolis homology and its description is often fairly straightforward. 

 On the other hand, for $s>0$, $\PP_t^s$ is not a primitive element of $\Steen$. In particular, 
  \begin{align*}
  \Delta(\Pt) &= \Pt \otimes 1 + \Q_1 \otimes \Q_1 + 1\otimes \Pt
   \end{align*}
 and its action on $H_*(X)$ for a ring spectrum $X$, does not follow the Leibniz rule. Instead, we have
\begin{equation} \label{coprodP}
    \Pt(xy) = \Pt(x)y + \Q_1(x)\Q_1(y) + x \Pt(y).
\end{equation}
As a result, the product of two $\Pt$ cycles may not necessarily be a $\Pt$ cycle, hence $\MM(X, \Pt)$ may not admit any multiplicative structure even if $X$ is a ring spectrum.  This is the main reason why the $\Pt$ Margolis homology calculations are significantly more complicated. %Moreover, unlike in the case of $\Q_i$ Margolis homology, the K{\"u}nneth isomorphism does not hold for  $\Pt$ Margolis homology.

Let us now consider the spectrum $\tmf$. It is well-known (\cite{HopkinsMahowald}, \cite{Mathew_tmf}) that  $$H_*(\tmf; \Ft) \iso \Ft[\zeta_1^8, \zeta_2^4, \zeta_3^2, \zeta_4, \zeta_5, \dots] \subset \Steen_*$$
is a subalgebra of $\Steen_*$. Here the elements $ \zeta_i$ are the images of  $\xi_i$ under the antipode of the Hopf algebra $\Steen_*$ (see Section \ref{Sec:Margolis}). 
The right action of $\Q_i$ is given by the formula (see \cite[\S2]{C} for details)
$$\Q_i (\zeta_n) = \zeta_{n-i-1}^{2^{i+1}}.$$ 
Then, since the $\Q_i$ are derivations, it can be easily seen that 
\begin{equation}
\MM(\tmf, \Q_0) = \Ft[\zeta_1^8, \zeta_2^4]
\end{equation}
\begin{equation}
\MM(\tmf, \Q_1) = \frac{\Ft[\zeta_1^8, \zeta_3^2, \zeta_4^2, \dots]}{\langle \zeta_3^4, \zeta_4^4, \dots\rangle}
\end{equation}
\begin{equation}
\MM(\tmf, \Q_2) = \frac{\Ft[\zeta_2^4, \zeta_3^2, \zeta_4^2, \dots]}{\langle \zeta_2^8, \zeta_3^8, \zeta_4^8, \dots\rangle}.
\end{equation}
In this paper, we give a complete calculation of  $\MM(\tmf^{\sma r}, \PP_2^1)$ for arbitrary $r \geq 1$. In fact, the calculation for $r >1$ follows from the case $r =1$, because after forgetting the internal grading one can construct a non-canonical isomorphism (see Section \ref{sec:misc})
$$\MM(\tmf^{\sma r}, \PP_2^1) \cong \MM(\tmf, \PP_2^1). $$  

 For the case $r =1$,  we give an iterated algorithm (see Definition~\ref{defn:name}) that constructs an $\Ft$-basis of $\MM(\tmf, \Pt)$.  We give a complete description of $\MM(\tmf, \Pt)$ in Theorem~\ref{thm:basistmf} which is the main result of this paper. Although $\MM(\tmf, \PP_2^1)$ is not an algebra, we notice that  $\MM(\tmf, \PP_2^1)$ is a module over an infinitely generated exterior algebra $\mathcal{S}$ (see Lemma~\ref{lem:S} for a description of $\mathcal{S}$). Theorem~\ref{thm:basistmf} also describes $\MM(\tmf, \Pt)$ as an $\SS$-module. 
%The key idea in this paper is the introduction of a filtration on  $H_*(\tmf, \mathbb{F}_2)$ which  we call the `length filtration'. The length filtration results in a spectral sequence, which we call the length spectral sequence.  We notice that the $E_2$ page is an algebra because $d_0$ follows Leibniz rule. However, $d_2$ does not follow Leibniz rule.

The key tool we use is the \emph{length spectral sequence}  \eqref{LSS}, which we define in Section \ref{Sec:Margolis}.
The length spectral sequence admits a $d_0$ differential and a $d_2$ differential and collapses at the $E_3$ page. The Leibniz rule does hold for the $d_0$, but not for $d_2$. In order to work around this issue, we notice that the $E_2$ page admits an action of $\SS$ (i.e. $d_2$ are $\SS$ linear) and we use it to simplify the computation of $E_{\infty}=E_3$.

% A careful analysis, namely Lemma~\ref{d2linear}, allows us to show that $\SS$ acts on the $E_2$ page and this action allows us to simplify the computation of the $E_3$ page. 
%and let's us conclude that $\MM(\tmf, \Pt)$ admits an $\SS$-module. 

We also notice that almost identical calculations lead to a complete description of  $\MM( (B\mathbb{Z}/2^{\times n})_+, \Pt)$. The methods developed in this paper can be considered as a blueprint for computations of $P_t^1$ Margolis homology of a variety of other $\Steen$-modules. 

Our calculations of $\MM(\tmf^{\sma r}, \Pt)$ have many applications, as the spectrum $\tmf$ has a wide range of applications, particularly in chromatic homotopy theory. First note that the cohomology of $\tmf$, as a module over the Steenrod algebra $\Steen$, is isomorphic to (see \cite{HopkinsMahowald}, \cite{Mathew_tmf})
\begin{equation} \label{eqn:homologytmf}
 H^*(\tmf; \mathbb{F}_2) \cong \Steen \modmod \Steen(2) 
\end{equation}
 where $\Steen(2)$ is the subalgebra of $\Steen$ generated by $\Sq^1, \Sq^2$ and $\Sq^4$. This, and a change of rings isomorphism, imply that the $E_2$ page of the Adams spectral sequence converging to $\tmf_* X$ (for a spectrum $X$) is 
\begin{equation} \label{tmfass}
 E_2^{s,t} := Ext^{s,t}_{\Steen(2)}(H^*(X), \Ft).
 \end{equation}
One can detect infinite families in the $E_2$ page via the map $$q:Ext^{s,t}_{\Steen(2)}(H^*(X), \Ft) \to Ext^{s,t}_{\Lambda(\Pt)}(H^*(X), \Ft).$$
 The codomain of $q$ can be understood by calculating $\MM(X, \Pt)$. Note that  $$Ext^{s,t}_{\Lambda(\Pt)}(\Ft, \Ft) \iso \Ft[h_{2,1}],$$  where $|h_{2,1}| =(1,6)$ and 
 \[ \Ft[h_{2,1}] \otimes \MM(X, \Pt) \subset  Ext^{s,t}_{\Lambda(\Pt)}(H^*(X), \Ft) \] 
 accounts for all the elements with positive $s$ filtration. This shows that the knowledge of $\MM(X, \Pt)$ is crucial in detecting patterns in the $E_2$-page of \eqref{tmfass}. 

\subsection*{Motivation \rom{1} - Towards homotopy groups of $K(2)$-local sphere }  
Computation of the homotopy groups of $L_{K(n)}S^0$ --- the sphere spectrum localized with respect to Morava $K$-theories $K(n)$ at various primes $p$ and heights $n$ --- is the central question of chromatic homotopy theory. It is sometimes easier to compute $\pi_*L_{K(n)}X$ for finite complexes other than the sphere, although very little data like this is known at $n=p=2$ anyway.  Recently, Bhattacharya and Egger introduced a family of finite spectra $Z$ \cite{BhattacharyaEgger}, and $\pi_*L_{K(2)}Z$ has been computed \cite{BBBCX,BhattacharyaEgger2}, the first example of a finite complex at $p=2$ whose $K(2)$-local homotopy groups are completely determined. The finite complex $Z$ can be constructed from the sphere spectrum, by a succession of cofiber sequences of self-maps (see \cite{BhattacharyaEgger}, the last one of which is 
 \[ \Sigma^5 A_1 \sma C\nu \overset{w}\to A_1 \sma C\nu \to Z .\]
In a quest to leverage the knowledge of $\pi_*L_{K(2)}Z$ to $\pi_*L_{K(2)}S^0$, one must first attempt to compute the $K(2)$-local homotopy groups of $A_1 \sma C\nu$. Very briefly, our strategy is to use the $v_2$-local $\tmf$-based Adams spectral sequence 
\[ E_1^{r,t} = v_2^{-1}\pi_{t}(\tmf \sma \overline{\tmf}^{\sma r} \sma A_1 \sma C\nu) \Longrightarrow \pi_{t-r}(L_{K(2)} A_1 \sma C\nu) \]
and compare it with that of $Z$. One can identify the $E_1$-page of the above spectral sequence using the classical Adams spectral sequence 
\begin{equation} \label{eqn:assa1cnu}
 E_2^{s,t} = \Ext^{s,t}_{\Steen}(H^{*}(\tmf \sma \overline{ \tmf}^{\sma r} \sma A_1 \sma C\nu), \Ft) \Rightarrow \pi_{t-s} (\tmf \sma \overline{\tmf}^{\sma r} \sma A_1 \sma C\nu ). 
\end{equation}
Because of \eqref{eqn:homologytmf} and the fact that $H^{*}(A_1 \sma C\nu) \iso \Steen(2) \modmod \Lambda(\Q_2, \Pt),$ and the change of rings isomorphism,  the $E_2$-page of the spectral sequence  \eqref{eqn:assa1cnu} has the form  
\[\Ext^{s,t}_{\Lambda(\Q_2, \Pt)}(H^{*}(\overline{\tmf}^{\sma r}), \Ft)
\] 
Hence, computation of $\MM(tmf^{\wedge r}, \Pt)$ is essential for understanding the $E_2$-page of \eqref{eqn:assa1cnu}. 
 
 \subsection*{Motivation \rom{2} - $\tmf$ resolution of the sphere spectrum} The connective spectrum $\mathit{bo}$ is not a flat ring spectrum, hence the $E_2$ page of the $\mathit{bo}$-based Adams spectral sequence does not have a straightforward expression like the classical Adams spectral sequence. However,  Lellmann and Mahowald \cite{LellmannMahowald} were able to calculate the $d_1$ differentials (also see \cite{BBBCX2}) and gave a description of the ``$v_1$-periodic part'' of the $E_2$-page. They identified the free Eilenberg--MacLane summand of $\mathit{bo}^{\sma r}$. To identify this free summand one needs to identify the $\Steen(1)$ free summand of $$H^{*}(\mathit{bo}^{\sma r}) \iso \Steen \modmod \Steen(1)^{\otimes r}.$$
 This can be done by calculating $\MM(\mathit{bo}^{\sma r}, \Q_0)$ and $\MM(\mathit{bo}^{\sma r}, \Q_1)$ and using the following theorem due to Margolis.
 \begin{thm}[{\cite[Chapter 19, Theorem 6]{Margolis_book}}] \label{mar} 
 An $\Steen(n)$-module $K$ is free if and only if $\MM(K, \PP_t^s) = 0$ whenever $s + t \leq n+1$ with $s<t$. 
 \end{thm}
 To emulate the strategy of Lellmann and Mahowald to understand the $\tmf$-based Adams spectral sequence for $S^0$ one needs to first identify the $\Steen(2)$-free part of 
 \[ H^{*}(\tmf^{\sma r}) \cong (\Steen \modmod \Steen(2))^{\otimes r} .\]
 Potentially, this can be identified using the knowledge of $\MM(\tmf^{\sma r}, \Q_i)$ for $i = 0,1,2$ and $\MM(\tmf^{\sma r}, \Pt)$, along with  Theorem~\ref{mar}.
 
\subsection*{Motivation \rom{3} - Infinite loop space of $\tmf$}
There are $\Steen$-modules $J(k)$, called Brown--Gitler modules \cite{BrownGitler}, which assemble into a doubly graded $\Steen$-algebra, denoted here by $J(\ast)^\ast$. Moreover, there is an $\Steen$-module isomorphism $J(\ast)^\ast \cong \F_2[x_1, x_2, \dots]$ where $x_i \in J(2^i)^1$ and the left $\Steen$ action on $J(\ast)^\ast$ is \cite{Schwartz} $$\Sq(x_i) = x_i + x_{i-1}^2.$$ In fact, $J(k)^\ast$ can be thought of as inheriting this action by virtue of being a subobject of $\Steen$. Because of this, minor modifications to methods of this paper apply to the calculation of $\MM(J(k) , \Pt)$. By \cite{KuhnMcCarty} there is a spectral sequence, obtained by studying Goodwillie towers, relating the knowledge of $H_*(\tmf ; \F_2)$ to that of  $H_*(\Omega^\infty \tmf ; \F_2)$ (also see \cite{HaugsengMiller} which provides a spectral sequence relating the cohomology of $\tmf$ to the cohomology of its infinite loop-space  $H^*(\Omega^{\infty} tmf; \F_2)$). Roughly speaking, this relies on computing certain derived functors, usually labeled $\Omega_s^\infty$, in the category of unstable modules over $\Steen$. It turns out that there is an isomorphism (see \cite{Goerss_UnstableProjectives} or \cite{HunterKuhn})
$$\Omega_s^\infty \Sigma^{-t} (\Steen\modmod \Steen(2))_\ast \cong \Ext^{s,t}_{\Steen(2)}(\F_2, J(\ast)),$$ 
so that these computations require an understanding of the $J(k)$ as modules over $\Steen(2)$, the hardest part of which is understanding how $\Pt$ acts.

 \subsection*{Organization of the paper}
 In Section~\ref{Sec:Margolis}, we recall some facts about the Steenrod algebra and its dual. We introduce the spectral sequence  \eqref{LSS}, which computes the $\Pt$ Margolis homology of $\tmf$, and discuss the $d_0$ differentials in it.  

 In Section~\ref{Sec:reduced}, we compute the $E_3=E_{\infty}$ page of the spectral sequence \eqref{LSS}. 
% we introduce the notion of reduced length, denoted $\ell$. We identify the subgroup $\SS$ of $\R$ which consists of some permanent cycles and show that $\MM(\R, \Pt)$ is a module over the ring $
%\SS$. We explicitly identify the generators of $\MM(\tmf, \Pt)$ as a module over $\SS$.
 We do that by introducing building blocks $M_J$ and computing $\MM(M_J, \Pt)$. Then we establish the relationship between $\MM(\tmf, \Pt)$ and  $\MM(M_J, \Pt)$ in Theorem~\ref{thm:basistmf}. 
% We also use the knowledge of $\mathcal{B}_J$ to provide a basis for $\MM(\tmf, \Pt)$, which is also a part of Theorem~\ref{thm:basistmf}. 
 
 In Section~\ref{sec:misc}, we show how to apply the same methods to calculate $\Pt$ Margolis homology for $\tmf^{\sma r}$ and $(B\mathbb{Z}/2^{\times k})_+$.
%  It turns out that the calculations are very similar to that of $\tmf$ and 
  Theorem~\ref{thm:basistmf} essentially gives complete answer in these cases.
 
   \subsection*{Acknowledgments}
 The authors are indebted to Nicolas Ricka for many insightful conversations. We are grateful to Nick Kuhn, Haynes Miller and John Rognes for their help and comments. We would like to thank the anonymous referee for many helpful comments and suggestions for improvement. This material is based upon work supported by the National Science Foundation Grant DMS-1440140, while the second author was in residence at the MSRI during the Spring 2019 semester, and Grants DMS-1638352 and DMS-2005627.

\section{Action of \texorpdfstring{$P_2^1$}{P21} and the length spectral sequence} \label{Sec:Margolis}
%\subsection{Action of $\Pt$ on $H_*(\tmf)$}
The dual Steenrod algebra $\Steen_*=\pi_*(H\FF_2 \wedge H\FF_2)$ has the structure of a  graded commutative  algebra which Milnor \cite{Milnor_Steenrod} 
showed to be a polynomial algebra
\[ \Steen_* \cong \Ft[\xi_1, \xi_2, \xi_3, \dots]\]
where $|\xi_i| = 2^{i} - 1$.   
Milnor defined $\Sq(r_1, r_2, \ldots) \in \Steen$ as the dual of  $\xi_1^{r_1}\xi_2^{r_2}\ldots$ and showed that they form an $\FF_2$ basis of the Steenrod algebra $\Steen$, known as the Milnor basis. The $\PP_t^s$ elements are defined as  
\[\PP_t^s=\Sq(r_1, \ldots), \text{  where  } r_i=
\begin{cases}
0, &i\neq t\\
2^s, &i=t.
\end{cases}
\]
 The action of an element $a \in \Steen$  on an $\Steen$-algebra follows the product rule given by the Cartan formula, i.e. 
 \[
 a(x\cdot y) = \Sigma_i a_i'(x) \cdot a_i''(y).
 \]
 where $\Delta(a)=\Sigma_i a_i'\otimes a_i''$ is the comultiplication in the Hopf algebra $\Steen$.
 
\begin{rem}
We would like to note that standard commonly used notation for the generators of the dual Steenrod algebra at $p=2$ differs from the notation in the original paper \cite{Milnor_Steenrod}, and we are grateful to John Rognes for explaining this to us. 
In \cite[Appendix 1]{Milnor_Steenrod}, Milnor denotes the polynomial generators of the dual Steenod algebra at $p=2$ by $\zeta_i$, so that 
$
\Steen_* \cong \FF_2[\zeta_1, \zeta_2, \ldots]
$
 and defines $\Sq(r_1, r_2, \ldots)$ as dual to the element $\zeta_1^{r_1}\zeta_2^{r_2}\cdots$. It has since become standard in the literature \cite{MosherTangora, A74, Margolis_book} to use a different notation and to denote the polynomial generators which were denoted by $\zeta_i$ in \cite[Appendix 1]{Milnor_Steenrod} by $\xi_i$, in order to match the notation for the odd primary Steenrod algebra. Hence in current standard notation $\Sq(r_1, r_2, \ldots)$ is dual to $\xi_1^{r_1}\xi_2^{r_2}\cdots$. The symbol $\zeta_i$ is now usually used to denote the image of $\xi_i$ under the antipode of the Hopf algebra $\chi: \Steen_* \to \Steen_*$, induced by the `flip map' on $H\FF_2 \wedge H\FF_2$. The elements $\zeta_i=\chi(\xi_i)$ can be computed recursively using the formula 
$\displaystyle \sum_{i+j=k}\xi_i^{2^j} \chi(\xi_j)=0$, together 
with the assumption that $\xi_0 = 1$ and $\xi_{i} = 0$ when $i <0$. 
%  This is not just a difference of notations, since the comultiplication on $\xi_i $ and $\zeta_i$ is very different. Namely,
%\begin{equation}\label{eq:Delta-definition}
%\Delta(\xi_k)=\sum_{i+j=k}\xi_i^{2^j}\otimes \xi_j
%\end{equation}
%while the comultiplication on $\zeta_i$ is given by 
%\begin{equation}\label{eq:Delta-definition2}
%\Delta(\zeta_k)=\sum_{i+j=k} \zeta_j \otimes \zeta_i^{2^j}
%\end{equation}
%which can be deduced from \eqref{eq:chi-definition} and \eqref{eq:Delta-definition}.
  %The anti-automorphic image  of $\chi(\xi_i)$ is usually denoted by $\zeta_i$. 
\end{rem}
%Let us now consider the spectrum $\tmf$.
  The homology of $\tmf$ is the subalgebra of $\Steen_*$ (\cite{HopkinsMahowald}, \cite[Theorem 5.13]{Mathew_tmf})
\begin{equation}\label{eq:homology-of-tmf}
\mathfrak{T}:=H_*(\tmf; \FF_2)\cong (\Steen \modmod \Steen(2))_* =\Ft[\zeta_1^8, \zeta_2^4, \zeta_3^2, \zeta_4, \zeta_5, \dots].
\end{equation}
Thus the action of $\Steen$ on $\mathfrak{T}$ is simply the restriction of the action of $\Steen$ on $\Steen_*$.

The right action of $\Steen$ on $\Steen_*$ is determined by the action of the total squaring operation $\Sq = 1 + \sum_{i >0}\Sq^i$  \cite[Lemma 3.6]{Pearson}
\begin{equation} \label{rightaction}
(\zeta_i)\Sq = \zeta_i + \zeta_{i-1}^2 + \zeta_{i-2}^4 +  \dots + \zeta_{1}^{2^{i-1}} + 1
\end{equation}
which is a ring homomorphism. 
\begin{rem}[Action of the total squaring operation]There are multiple ways to define the action of $\Steen$ on $\Steen_*$. While we will be using the action defined by \eqref{rightaction}, we would like to collect other commonly used actions here. By \cite{Mah81}, the right and left actions of $\Sq$ on $\xi_i$ are given by the formulas
\begin{align*} \Sq(\xi_i) &= \xi_i + \xi_{i-1}^2 \\
(\xi_i)\Sq &= \xi_i + \xi_{i-1},
\end{align*}
while the left action on $\zeta_i$ is 
\[ \Sq(\zeta_i) = \zeta_i + \zeta_{i-1} + \dots + \zeta_1 + 1.\]
From these formulas we can derive 
\begin{align*}
    \Q_{i-1}(\xi_n)&=\xi_{n-i}^{2^i}\\
    (\zeta_n)\Q_{i-1} &=\zeta_{n-i}^{2^i};
\end{align*}
the second equation can also be found in \cite{C}.
\end{rem}

\begin{impnotn} \label{left-right}
Since we only work with the right action of $\Sq$ in this paper, we will write $a(x)$ to denote the \emph{right} action of $a\in \Steen$ on $x\in H_*(\tmf)$ for the rest of the paper. Thus, from now on
\[
\boxed{
\textcolor{red}{a(x):=(x)a.}
}\]
\end{impnotn}

We now focus on the action of $\Pt = \Sq(0,2)=\Sq^2 \Sq^4 + \Sq^4 \Sq^2$ on $\mathfrak{T}$. From \eqref{rightaction}, one can easily see that $\Sq^{2i}$ acts trivially on $\zeta_n$, when $i>0$ and $n\neq 1$. It follows immediately that  $$\Pt(\zeta_i) = 0.$$
 Beware! This \emph{does not} mean that $\Pt( \zeta_i \zeta_j) = 0$, as the Leibniz rule does not hold. Since
$ 
\Delta(\Pt) = \Pt \otimes1 + \Q_1\otimes \Q_1 + 1\otimes \Pt,$
we obtain the product formula  
\begin{equation}\label{eq:Leibniz-on-p21}
 \Pt(xy) = \Pt(x)y + \Q_1(x)\Q_1(y) + x \Pt(y).
\end{equation}
Using  $\Q_1(\zeta_i)=\zeta_{i-2}^4$, we  get  
\begin{equation} \label{prodex1}
\Pt(\zeta_i \zeta_j) = \zeta_{i-2}^4 \zeta_{j-2}^4
\qquad \qquad \qquad
 \Pt(\zeta_i^2) = \zeta_{i-2}^8. 
 \end{equation}
%\begin{equation} \label{prodex1}
%\Pt(\zeta_i \zeta_j) = \zeta_{i-2}^4 \zeta_{j-2}^4
%\end{equation}  
%\begin{equation} \label{prodex2}
% \Pt(\zeta_i^2) = \zeta_{i-2}^8. 
% \end{equation}
Formulas become more complicated for triple products, e.g. 
\[ \Pt(\zeta_i \zeta_j \zeta_k) = \zeta_{i-2}^4 \zeta_{j-2}^4 \zeta_k + \zeta_{i-2}^4 \zeta_j \zeta_{k-2}^4 + \zeta_i \zeta_{j-2}^4 \zeta_{k-2}^4,\]
and in general we have the following result.
\begin{lem} \label{Lem:P21_action}
The action of $\Pt$ on $\mathfrak{T}$ is given by the formula 
\begin{eqnarray*}
\Pt(\zeta_{i_1} \dots \zeta_{i_n}) &=& \sum_{1 \leq j< k \leq n} \frac{\zeta_{i_1} \dots \zeta_{i_n}}{\zeta_{i_j}\zeta_{i_k}}  \Q_1(\zeta_{i_j})\Q_1(\zeta_{i_k}) \\
&=&\sum_{1 \leq j< k \leq n}  \zeta_{i_1}\dots \zeta_{i_{j-1}} \zeta_{i_j -2}^4 \zeta_{i_{j+1}}\dots \zeta_{i_{k-1}} \zeta_{i_k - 2}^4  \zeta_{i_{k+1}} \dots \zeta_{i_n}, 
\end{eqnarray*}
where indices are allowed to repeat. 
\end{lem}
\begin{proof}
Follows from an inductive argument on $n$, using \eqref{eq:Leibniz-on-p21} and the facts that $\Pt(\zeta_i)=0$ and $\Q_1(\zeta_i)=\zeta_{i-2}^4$. 
\end{proof}

The technique developed in this paper begins with the following observation. Consider the subalgebra  
\[ \mathcal{E} := \Ft[\zeta_1^8, \zeta_2^4, \zeta_3^2, \zeta_4^2, \zeta_5^2, \dots ]  \subset \mathfrak{T}=\Ft[\zeta_1^8, \zeta_2^4, \zeta_3^2, \zeta_4, \zeta_5, \dots]\]
 which we will call the \emph{even} subalgebra of $\mathfrak{T}$, as every element in $\mathcal{E}$ has even grading.  Since $|\Q_1| =3$ and every element in $\mathcal{E}$  has even grading, $\Q_1$ must act trivially on $\mathcal{E}$.  Thus, $\Pt$ restricted to $\mathcal{E}$ follows the Leibniz rule, hence $\MM(\mathcal{E}^{\otimes r}, \Pt)$ is an algebra. Using \eqref{prodex1} and the K{\"u}nneth isomorphism for a derivation, we can easily deduce the following result.
 \begin{lem} The $\Pt$ Margolis homology of $\mathcal{E}$ is given by
\[ \MM(\mathcal{E}, \Pt) \iso \Lambda(\zeta_2^4, \zeta_3^4, \zeta_4^4,\dots).\]
Moreover 
\[ \MM(\mathcal{E}^{\otimes r}, \Pt)  \iso  \MM(\mathcal{E}, \Pt)^{\otimes r}  \iso (\Lambda(\zeta_2^4, \zeta_3^4, \zeta_4^4,\dots))^{\otimes r}.\]
\end{lem}
Now consider
%$
%0 \to \mathcal{E} \to \mathfrak{T} \to \mathcal{K} \to 0
%$,
the quotient $\mathcal{K} := \frT \modmod \mathcal{E} \cong \FF_2 \otimes_{\mathcal{E}} \mathfrak{T}$. We have an isomorphism  $\mathcal{K}\cong \Lambda( \zeta_4, \zeta_5, \dots)$, and the induced action of $\Q_1$ and $\Pt$ on $\mathcal{K}$ is trivial.
For a set $A$, we let $\Ft\langle A \rangle$ denote the $\Ft$-vector space which has the generating set $A$.
 The algebra $\mathcal{K}$ admits a natural increasing  filtration 
\[ G^p(\mathcal{K}) := \mathbb{F}_2\langle \zeta_{i_1} \dots \zeta_{i_k} | k \leq p \rangle.\]
induced by the length of the monomials. We call it the \emph{length filtration}. 

This length filtration on $\mathcal{K}$ induces an increasing filtration $\{ G^p(\frT)\}_{p\geq 0 } $ on $\frT$,
where  $G^p(\frT)$ is the pullback of $G^p(\mathcal{K})$ (in vector spaces) along the quotient map $\frT \twoheadrightarrow \mathcal{K}$
\[ 
\xymatrix{
 G^p(\frT) \ar[r] \ar[d] & \frT \ar@{->>}[d] \\
G^p(\mathcal{K}) \ar[r] & \mathcal{K}.
}
\]
\begin{defn}
Let $I$ be a finite tuple of natural numbers, and for $I= \set{i_1, \dots, i_n}$ let $\zeta^I$  denote the monomial $\zeta_1^{i_1}\dots \zeta_n^{i_n}.$
Then the \emph{length} $L$ of $\zeta^I$ is defined by 
\[
   L(\zeta^I)= \sum_{j=1}^{|I|} (i_j \mod 2).
\]
\end{defn}
In other words, $L(\zeta^I)$ counts the number of odd exponents  in $\zeta^I.$ Then $G^p(\frT)$ is the span of monomials $\zeta^{I}$ of length less than or equal to $p$ 
\[ G^p(\frT) \cong \mathbb{F}_2\langle \zeta^I | L(\zeta^I) \leq p \rangle.\]
The  length function $L$ measures ``how far"  a given monomial in $\mathfrak{T}$ is from the even subalgebra $\mathcal{E}$.
%\begin{rem}
%The  length function $L$ measures ``how far"  a given monomial in $\mathfrak{T}$ is from the even subalgebra $\mathcal{E}$.
%\end{rem}
% The notion of length leads to an increasing filtration of $\mathfrak{T}$, $\lbrace G^p: p \in \mathbb{N} \rbrace$, called the \emph{length filtration}, where  
%\[
  % G^p(T)=\langle  \zeta^I  L(\zeta^I)\leq p \rangle.
%\]
%This filtration was recently used in \cite{C} and is essentially the filtration introduced in \cite{CarlisleKuhn}. 
Since there is an  $\Ft$-vector space isomorphism 
\[ \mathfrak{T} \cong \mathcal{E} \otimes \frT \modmod \mathcal{E} = \mathcal{E} \otimes \mathcal{K} \]
 any monomial  $m \in \mathfrak{T}$ can be uniquely written as $e \cdot k$ where $e \in \mathcal{E}$ and $k \in \mathcal{K}$. 
 \begin{ex}
If $m = \zeta_3^4 \zeta_5^5 \zeta_8^3$, then there is an unique expression $m = e \cdot k $ ,where  $e = \zeta_3^4 \zeta_5^4 \zeta_8^2 \in \mathcal{E}$ and $k = \zeta_5 \zeta_8 \in \mathcal{K}$. 
\end{ex}
The following lemma shows that the action of $\Q_1$ and $\Pt$ preserves the length filtration. 
\begin{lem}\label{Lem:action-reduces-length} Let $m \in \mathfrak{T} $ be any monomial. 
\begin{enumerate} 
\item If $m\in \mathcal{E}$, then $\Q_1(m)=0$ and $\Pt(m)\in \mathcal{E}$. 
\item If $m\notin \mathcal{E}$, then $\Q_1(m) \in G^{L(m) -1}(\mathfrak{T})/G^{L(m) -2}(\mathfrak{T})$ and 
$$\Pt(m)= m_{L}+m_{L-2},$$ where $m_{L} \in G^{L(m)}(\mathfrak{T})/G^{L(m) -1}(\mathfrak{T})$ and $m_{L-2} \in G^{L(m) -2}(\mathfrak{T})/G^{L(m)-3}(\mathfrak{T})$.
%If $\Pt(m) \neq 0$, then 
%\[L(\Pt (m))=
%\begin{cases}
%L(m), &\text{ if }m\in E\\
%L(m)-2, &\text{ otherwise. }
%\end{cases}
%\]
\end{enumerate}
\end{lem}
\begin{proof} When $m \in \mathcal{E}$,  $\Q_1(m) = 0$ by the Leibniz rule. Using Lemma~\ref{Lem:P21_action} we have 
$\Pt(m) \in \mathcal {E}$ and
$ L(\Pt(m)) = L(m) =0 $.
 %completing the proof of  $(i)$. 

 Now assume $m\notin \mathcal{E}$, which means $m= e\cdot k$ for some $e \in \mathcal{E}$ and some $1 \neq k\in \mathcal{K}$. Note that $k$ is of the form $\zeta_{i_1} \dots \zeta_{i_n}$,  where indices do not repeat.

The action of $\Q_1$ is given by the formula 
\[ \Q_1(\zeta_{i_1} \dots \zeta_{i_n}) = \sum_{k=1}^n \zeta_{i_1} \dots \zeta_{i_{k-1}} \zeta_{i_k -2}^4 \zeta_{i_{k+1}} \dots \zeta_{i_n}\]
where we allow repetition of indices.  Since $\Q_1$ acts trivially on $\mathcal{E}$, it follows that 
\[ \Q_1(e \cdot k) = e \cdot \Q_1(k). \]
 From the formula above we see that $\Q_1(k) \neq 0$ and $L(\Q_1(k)) = L(k) -1$. Hence, 
\[ L(\Q_1(m)) = L(e \cdot \Q_1(k)) = L(\Q_1(k))= L(k) -1 = L(e \cdot k) -1= L(m) -1. \]

Next, note that 
\[
\Pt(m)= \Pt(e) \cdot k + \Q_1(e) \cdot \Q_1(k) + e \cdot \Pt(k) = \Pt(e) \cdot k+ e \cdot \Pt(k)
 \]
From Lemma~\ref{Lem:P21_action}, we see that $L(\Pt(k)) = L(\Pt(k)) - 2$ assuming $\Pt(k) \neq 0$. Now set 
$m_L = \Pt(e) \cdot k$ and $m_{L-2} = e \cdot \Pt(k)$
%, and hence 
%\[ L(\Pt(m)) = L(e \cdot \Pt(k)) = L(\Pt(k))= L(k)-2= L(e \cdot k) -2 =L(m) -2.\]
\end{proof}
\begin{lem}\label{lem:P21Q1=Q1P21}
The Hopf algebra $ \Lambda(\Q_1, \Pt)$ is commutative and cocommutative.
 \end{lem}
 \begin{proof} Commutativity follows from the fact that $\Pt$ and $\Q_1$ commute  \cite[Lemma~1.3(2)]{AdamsMargolis} (in the notation of \cite{AdamsMargolis}, $\Pt = P_2(2)$ and $\Q_1 = P_2(1)$).  Cocommutativity follows from the fact that  
 the diagram
 \[ 
 \xymatrix{
  \Lambda(\Q_1, \Pt) \ar[r]^-{\Delta}\ar[rd]_-{\Delta} &  \Lambda(\Q_1, \Pt) \otimes  \Lambda(\Q_1, \Pt) \ar[d]^{\text{flip}} \\
 &  \Lambda(\Q_1, \Pt) \otimes  \Lambda(\Q_1, \Pt)
 }
 \]
 commutes, because of \eqref{coprodQ} and \eqref{coprodP}.
 \end{proof}

If $M$ is a $\Lambda(\Q_1, \Pt)$-module then let $\mathcal{C}_M^{\bullet}$ denote the periodic chain complex 
\[ \dots \overset{\Pt}{\longrightarrow} M  \overset{\Pt}{\longrightarrow} M \overset{\Pt}{\longrightarrow} \dots . \]
Its homology groups are isomorphic in each degree, i.e. 
\[ H_i(\mathcal{C}_M^{\bullet}) \iso H_j(\mathcal{C}_M^{\bullet})\] 
for all $i,j \in \mathbb{Z}$.  We use $\MM(M, \Pt)$ to denote this common homology group. When $M = \frT$, the filtration $G^{\bullet}(\frT)$ induces a filtration on $C^{\bullet}_{\frT}$. 
By Lemma~\ref{Lem:action-reduces-length}, $\Pt$ respects the length filtration. This means we have a short exact sequence of chain complexes 
\[
\xymatrix{
 0 \ar[r] &  \bigoplus_{p\in \mathbb{Z}} G^{p-1}( C^{\bullet}_{\frT}) \ar[r] & \bigoplus_{p\in \mathbb{Z}} G^{p }( C^{\bullet}_{\frT}) \ar[r] & \bigoplus_{p\in \mathbb{Z}} \frac{ G^{p }( C^{\bullet}_{\frT})}{G^{p-1}( C^{\bullet}_{\frT})} \ar[r]& 0.
 }  \]
Upon taking the homology, this short exact sequence of chain complexes produces an exact couple,  resulting in a  spectral sequence
\[ 
E_1^{p,q} := H^{q}\left( \frac{ G^{p }( C^{\bullet}_{\frT})}{G^{p-1}( C^{\bullet}_{\frT})}\right) \Rightarrow H^q(C^{\bullet}_{\frT}).    
\]
We rewrite this spectral sequence as 
\begin{equation} \label{LSS}
 E_1^{p} := \MM\left(\frac{G^{p}(\frT)}{G^{p-1}(\frT)}, \Pt\right)  \Rightarrow  \MM(\tmf, \Pt).
\end{equation}
and we call it the \emph{length spectral sequence}.
%\begin{rem} Internal grading $r$ in \eqref{LSS} do not play any role in calculations to follow. In other words, all our result remains true if we forget the internal grading. This is the reason why we would often neglect the internal grading. 
%\end{rem}

The $E_1$ page of \eqref{LSS} is easy to calculate.  Note that the associated graded 

\[ \bigoplus_{p\geq 0}\frac{G^{p}(\frT)}{G^{p-1}(\frT)} \cong  \mathcal{E} \otimes \mathcal{K} \]
as an $\FF_2$-algebra. 
The action of $\Lambda(\Q_1, \Pt)$ on $\mathcal{E}\otimes \mathcal{K}$ is defined using the Cartan formula as in the definition below.
\begin{defn}[\cite{Margolis_book}, p.186]\label{def:tensor-Hopf}
Let $\Gamma$ be any Hopf algebra. For two  $\Gamma$-modules $M$ and $N$,  the underlying $\Ft$ vector space of $M \otimes N$ is simply $M \otimes_{\mathbb{F}_2}N$, and $\Gamma$ acts via the diagonal map, i.e. 
\[ a(m \otimes n) = \sum_i a_i(m) \otimes a_i'(n) .\]
 where $a \in \Gamma$ and $\Delta(a) = \sum_{i} a_i \otimes a_{i}'$, where $\Delta$ is the coproduct of the Hopf algebra.
\end{defn} 

 Now we describe the action of  $\Pt$ on   a monomial $m \in \bigoplus_{p\geq -1} \frac{G^{p}(\frT)}{G^{p-1}(\frT)}$. 
  Write $m = e\otimes k$ for some $e \in \mathcal{E}$ and $k \in \mathcal{K}$. By
  Definition \ref{def:tensor-Hopf} %Lemma~\ref{Lem:action-reduces-length} that 
\[ \Pt(m) = \Pt(e \otimes k) = \Pt(e) \otimes k. \]
 Since the length filtration $G^{\bullet}( \frT )$ is multiplicative, i.e.
 \[  G^{p}( \frT ) \cdot G^{p'}( \frT )  \subset G^{p+p'}( \frT ), \]
 and $\Pt$ restricted to $\mathcal{E}$ follows the Leibniz rule, the $E_1$ page of \eqref{LSS} is an $\FF_2$-algebra and isomorphic to 
 \[ E_1^{\ast} \cong \MM(\mathcal{E} \otimes \mathcal{K}, \Pt) \cong \MM(\mathcal{E}, \Pt) \otimes \mathcal{K} \iso \Lambda(\zeta_2^4, \zeta_3^4, \dots )\otimes \Lambda(\zeta_4, \zeta_5, \dots).\]
  In order to avoid confusion regarding the multiplicative structure of $E_1^{\ast}$, it is convenient to rename the generators. 
  \begin{notn} \label{Notn:xt} We set  $x_{i}:=\zeta_{i+3}$ and  $t_i:=\zeta_{i+1}^4$.
  %so that we have
%\begin{equation}\label{eq:Q1-action-on-x-and-t}
% \Q_1(x_i) = t_i.
% \end{equation}
 Further, for  finite subsets $I= \set{i_1, \dots, i_n}  \subset \mathbb{N}$ and $J= \set{j_1, \dots, j_m}  \subset \mathbb{N}$, we let $t_I$ and $x_{I}$ denote the monomials $t_{i_1}\dots t_{i_n}$ and $x_{j_1}\dots x_{j_n}$ respectively.  We use  $t_Ix_J$ to denote the tensor product  $t_I \otimes x_J$.
\end{notn}

 Lemma~\ref{Lem:action-reduces-length}  and Lemma \ref{lem:P21Q1=Q1P21} imply that we have a commutative diagram of 
 chain complexes 
 \[
\xymatrix{
0 \ar[r] \ar@{=}[d] &  \bigoplus_p G^{p}( C^{\bullet}_{\frT}) \ar[r] \ar[d]_{\Q_1(\ )} & \bigoplus_p G^{p +1 }( C^{\bullet}_{\frT}) \ar[r] \ar[d]_{\Q_1(\ )}& \bigoplus_p \frac{ G^{p +1 }( C^{\bullet}_{\frT})}{G^{p}( C^{\bullet}_{\frT})} \ar[r] \ar[d]^{\Q_1(\ )}& 0 \ar@{=}[d].\\
 0 \ar[r] &  \bigoplus_p G^{p-1}( C^{\bullet}_{\frT}) \ar[r] & \bigoplus_p G^{p }( C^{\bullet}_{\frT}) \ar[r] & \bigoplus_p \frac{ G^{p }( C^{\bullet}_{\frT})}{G^{p-1}( C^{\bullet}_{\frT})} \ar[r] & 0.
 }  \]
 Consequently there is an action of $\Q_1$ on each page of \eqref{LSS}, which shifts the length filtration by $-1$. In particular, we note $\Q_1(x_i) = t_i $ and in general 
 \begin{equation} \label{Q1}
   \Q_1(t_Ix_J) =  \sum_{ j \in J} t_{j} t_I x_{J -\set{j}}.
   \end{equation}

Let $m \in \frT$ be any monomial, $m_L$ and $m_{L-2}$ be as in  Lemma~\ref{Lem:action-reduces-length}, and let $[m]$ denote the equivalence class in the $E_1$ page of \eqref{LSS} represented by $m$.  Lemma~\ref{Lem:action-reduces-length} %(along with general theory of spectral sequences)
 implies that the $d_1$ differential of \eqref{LSS} is trivial, 
 \[ d_2([m]) = [m_{L-2}]\]
 for the class of the monomial $m \in \frT$ in the $E_1$ page, and the spectral sequence \eqref{LSS} collapses at the $E_3$ page. If we write $m \in \frT$ as $m = e \cdot k$, where $e \in \mathcal{E}$ and $k \in \mathcal{K}$, then 
\[
  d_2([m]) = [e \cdot \Pt(k)] = [e] \cdot [\Pt(k)]. 
  \]
This means that the $d_2$ differential of \eqref{LSS} is $\MM(\mathcal{E}, \Pt)$-linear. It follows from the formula of Lemma~\ref{Lem:P21_action} that  
 \begin{equation} \label{d2} d_2(t_I x_J) = \sum_{K \in J[2]} t_K t_I x_{J -K},
 \end{equation}
where $J[2]$ is the set of subsets of $J$ which contain two elements.

 The formula for the $d_2$ differentials is intimately related to the action of $\Q_1$ on the $E_2$ page of \eqref{LSS}.  The $\Lambda(\Q_1)$-module structure on $E_2^{\bullet}$ (see  \eqref{Q1}) can extend it to the $\Lambda(\Q_1, \Pt)$-module structure using the algebra structure of $E_2^{\bullet}$ and the product formula \eqref{eq:Leibniz-on-p21}), along with the initial condition 
 \begin{equation} \label{initial}
 \Pt(x_i ) = \Pt (t_i) = 0.
 \end{equation}
% The initial condition \eqref{initial} implies that 
% \[ \Pt(x_i x_j ) = \Pt(x_i) x_j + \Q_1(x_i)\Q_1(x_j) + x_i \Pt(x_j) = t_i t_j\]
%when $i \neq j$.
The action of $\Pt$ that results from this procedure is 
\begin{equation} \label{P21actionE2}
 \Pt(t_I x_J) = \sum_{K \in J[2]} t_K t_I x_{J -K} 
 \end{equation}
on the monomial basis, which can be extended to all of $E_2^{\bullet}$ using $\FF_2$-linearity.
% Since the set of monomials form an $\FF_2$-basis of $E_2^{\bullet}$, \eqref{P21actionE2}  determines  the action of $\Pt$ on $E_2^{\bullet}$.  
 Notice that the action we obtain through this process coincides with the formula for the $d_2$ differentials \eqref{d2}.

\section{The reduced length} \label{Sec:reduced}

For convenience, we denote the $E_2$-page of \eqref{LSS} by
\[\mathcal{R}= \Lambda(t_i: i \geq 1) \otimes \Lambda(x_i: i \geq 1), \]
which is an $\FF_2$-algebra, 
 as well as a $\Lambda(\Q_1, \Pt)$-module, where the actions of $\Q_1$ and $\Pt$ are given by \eqref{Q1} and \eqref{P21actionE2} respectively.
  In this section we analyze the $\Lambda(\Q_1, \Pt)$-module structure of $\mathcal{R}$, which leads us to a  description of 
\[ E_{\infty}^{\bullet} \cong \dots \cong E_3^{\bullet} \cong H(E_2^{\bullet}, d_2) \cong \MM(\mathcal{R}, \Pt).  \]
The main idea here is to notice that the action of $\Pt$  is linear with respect to the subalgebra $$\SS := \Lambda(t_ix_i | i \in \mathbb{N}_+)  \subset \mathcal{R},$$  which implies that $\MM(\R, \Pt)$ admits an $\SS$-module structure. 
% In this section we notice that the action of $\Pt$  is linear with respect to the subalgebra $$\SS := \Lambda(t_ix_i: i \in \mathbb{N}_+)  \subset \mathcal{R}.$$ This results in the notion of reduced length (see Definition~\ref{reduceL}). Because of $\SS$-linearity, $\MM(\R, \Pt)$ admits an $\SS$-module structure. In Theorem , we completely describe the $\SS$-module structure of $\MM(\R, \Pt)$, which is also  the $E_{\infty}$ page of \eqref{LSS}.

\begin{lem} \label{lem:S} The subalgebra $\SS \subset \mathcal{R}$ is a  trivial $\Lambda(\Q_1, \Pt)$-submodule which splits off as a $\Lambda(\Q_1, \Pt)$-module. 
\end{lem}
\begin{proof} 
For any element $t_I x_I \in \SS$, it is clear from \eqref{Q1} and  \eqref{d2} that 
\[\Q_1(t_Ix_I)=  0=\Pt(t_I x_I).\] Thus $\SS$ is a trivial submodule.  

Now observe from \eqref{Q1} and  \eqref{d2} that none of the monomials  $t_Ix_I \in \SS$ is  a summand 
of $\Q_1(t_{I'} x_{J'})$ or  $\Pt(t_{I'} x_{J'})$ for any choice of $I'$ and $J'$.  Hence, $\SS$ is a split summand.
%Thus none of the elements of  $\SS$ is a target of a $d_2$ differential, hence nonzero in the $E_{3}$ page. Then the result follows from the fact that \eqref{LSS} collapses at the  $E_3$ page.
\end{proof}
\begin{cor}  \label{cor:S} 
 Every element of $\SS$ is a nonzero cycle in the $\MM(\R, \Pt)$.
\end{cor}
\begin{lem} \label{d2linear}The action of $\Pt$ on $\mathcal{R}$ is $\SS$-linear. 
\end{lem} 
\begin{proof} It is enough to show that  
\begin{equation} \label{check} \Pt(t_{i} x_{i} \cdot  t_{I} x_{J} ) = ( t_{i} x_{i} )\cdot \Pt(t_{I}x_{J}) .
\end{equation}
 If $i \in I$, then $t_{i}t_{I} = 0 $. Hence both the LHS and the RHS of \eqref{check} are zero.
 
  If $i \in J$, then $x_{i} x_{J} = 0$, hence LHS of \eqref{check} is zero. On the other hand, 
\begin{eqnarray*}
RHS &=&  t_{i} x_{i} \cdot  \sum_{ K \in J[2]} t_K t_I x_{J - K} \\
&=&  \sum_{i \in K \in J[2]} t_{i}t_Kt_I x_{i}x_{J-K}  + t_{i} \cdot( \sum_{ i \notin K \in J[2]} t_{i}t_I x_{i}x_{J-K})=0,
\end{eqnarray*}
as $t_{i} t_K =0$ when $i \in K$ and $x_{i} x_{J-K} = 0$ when $i \notin K$. 

Now consider the case when $i \notin I \cup J$.  Let $I' = I \cup \lbrace i \rbrace$  and $J' = J \cup \lbrace i \rbrace$. Then,  
\begin{eqnarray*}
\Pt(t_{i} x_{i} \cdot t_{I} x_J)=  \Pt(t_{I'}x_{J'}) &=&  \sum_{ K \in J'[2]} t_Kt_{I'} x_{J' -K}\\
&=&   \sum_{i \in K \in J'[2]} t_Kt_{I'} x_{J' -K} + \sum_{i \notin K \in J'[2]} t_Kt_{I'} x_{J' -K}\\
&=& \sum_{i \notin K \in J'[2]} t_Kt_{I'} x_{J' -K}   \\
&=& t_{i}x_{i} \cdot \sum_{ K \in J[2]} t_Kt_{I} x_{J-K}  \\
&=& t_{i}x_{i} \cdot \Pt(t_I x_J).
\qedhere
\end{eqnarray*}
\end{proof}
\begin{rem}
%[No product structure on the $E_3$ page]
 While the $E_2$ page of \eqref{LSS} admits an $\FF_2$-algebra structure, the $E_3$ page does not admit any multiplicative structure. This is because the $d_2$ differentials do not follow Leibniz rule and the product of $d_2$ cycles may not be a cycle. For  example,   $x_i$ for all $i \in \mathbb{N}_+$, is a $d_2$-cycle, whereas $x_ix_j$ for $i \neq j$ supports a differential $d_2(x_ix_j)= t_i t_j$ by \eqref{d2}.  Even if $\alpha$, $\beta$ and $\alpha \cdot \beta$ are $\Pt$ cycles it is unclear that the pairing 
 $[\alpha] \cdot [\beta] = [\alpha \cdot \beta ]$
 is well-defined in the $E_3$ page.
 \end{rem}
\begin{cor} \label{cor:S} 
 $\MM(\R, \Pt)$ is a module over the ring $\SS$.
\end{cor}
\begin{proof} 
By Lemma~\ref{d2linear}, there exists a pairing $\mu: \SS \otimes \R \to \R$ such that the diagram 
\[ 
\xymatrix{
\SS \otimes \R \ar[r]^{\mu} \ar[d]_{1 \otimes \Pt} & \R \ar[d]^{ \Pt}  \\
 \SS \otimes \R \ar[r]_{\mu} & \R.
}
\]
commutes.
It follows that $\MM(\R, \Pt)$ is an $\SS$ module. 
\end{proof}

As a result, we only need to understand the action of $\Pt$ on the generators of $\R$
when viewed as an $\SS$-module.
In order to approach this problem we introduce the notion of \emph{reduced length}. 
\begin{defn} \label{reduceL} For any monomial $t_Ix_J \in \R$ the \emph{reduced length} $\ell$ is
\[ \ell(t_Ix_J) =|J-I|=|J \cap I^c| = |J| - |J \cap I| ,
\]
where $I^c$ denotes the complement of $I$.
\end{defn}

Note that the length of $t_Ix_J\in \R$ is given by the formula $L(t_Ix_J) = |J|$; in other words, it is counting the number of factors of $x_J$. Whereas,  $\ell(t_Ix_J)$ counts only those factors $x_{j}$ in $x_J$ for which $t_{j}$ is not a factor of $t_I$.
For example,
 \[
 \ell(x_1)=\ell(t_1x_1x_2)=\ell(t_1t_2x_1x_2x_3)=\ell(t_1t_2t_3x_4)=1
 \]
 \[
 \ell(x_1x_2)=\ell(t_1x_1x_2x_3)=\ell(t_1t_2t_3t_4x_5x_6)=2.
 \]
  \begin{rem}
The  reduced length function $\ell$ measures ``how far"  a given monomial in $\mathcal{R}$ is from the  subalgebra $\SS$.
\end{rem}
   For each $i \in \mathbb{N}_{+} $, let $M_i := \Lambda(\Q_1)\{ x_i \} \subset \mathcal{R}$ denote the $\Lambda(\Q_1, \Pt)$-submodule isomorphic to $\Lambda(\Q_1)$ and generated by $x_i$.  
   %We denote $\Q_1x_i$ by $t_i$. Thus, as an $\FF_2$-vector space $M_i$ has two generators, $t_i$ and $x_i$, such that 
%\begin{equation}\label{eq:defining-M-i}
 % \Q_1(x_i)=t_i \qquad \quad \Q_1(t_i)=0 \qquad \quad \Pt(x_i)=\Q_1(t_i)=0.
  %\end{equation}
   For an indexing set $K \underset{\text{finite}}\subset \mathbb{N}_+$, let  
$$ M_K :=\bigotimes_{j \in K} M_j \subset \mathcal{R}$$
with the convention that $M_\emptyset := \FF_2$. If the indexing set is $[n]=\{1, ..., n\} \subset \mathbb{N}_+$, then we write  $M_{[n]}$ to denote $M_{\{1, ..., n\}}$. 
%In order to be explicit about the action of $\Lambda(\Q_1,\Pt)$ on $M_K$, it is worth reminding ourselves the definition of the tensor product over a Hopf algebra. 

 In Figure \ref{picbasic}, Figure~\ref{pictensor2} and Figure~\ref{pictensor3} we present $M_{i}$, $M_{\{ 1,2\}}$ and $M_{\{ 1,2,3\} }$ respectively as a $\Lambda(\Q_1, \Pt)$-module. In these figures the blue curved lines depict the action of $\Q_1$ and red boxed lines depict the action of $\Pt$.
 \begin{figure}[h]
\centering
\begin{tikzpicture}[scale = .5]
\node (a0) at (0,0) {$t_i $};
\node (a1) at (0,2) {$x_i$};
\draw[blue, bend left] (a0) to (a1);
\end{tikzpicture}
\caption{$M_{i}$ as a module over $\Lambda(\Q_1, \Pt)$}
\label{picbasic}
\end{figure}

Note that the set  $\W := \lbrace t_Ix_J \in \mathcal{R}| I \cap J = \emptyset \rbrace$  forms a generating set for $\R$ as an $\SS$-module as any  monomial $t_Ix_J \in \R$ can be uniquely written as a product of an element of $\W$ and a monomial in $\SS$:
 \[ t_Ix_J = t_{I \cap J} x_{I \cap J} \cdot t_{I - (I \cap J)}x_{J - (I \cap J)}.\] 
%where $t_{I \cap J} x_{I \cap J}\in \SS$ and $ t_{I -J}x_{J - I} \in \W$. 
For any finite subset $K \subset \mathbb{N}_+$, 
\[\W_K:= \set{ t_I x_J | I \cup J = K, I\cap J=\emptyset} \subset \W \]  
forms an $\FF_2$-basis for $M_K$, i.e. $\Ft\langle \W_K \rangle =  M_K. $ Since 
\[ \W = \bigsqcup \limits_{K \underset{\text{finite}}{\subset} \mathbb{N}_{+} } \W_K\] 
and 
$\Ft\langle \W_K\rangle = M_K$ is closed under the action of $\Q_1$ and $\Pt$ (these actions preserve the total indexing set $K$, by \eqref{Q1}  and \eqref{d2}), we learn that 
\[ \mathcal{R} \modmod \SS \cong  \FF_2 \otimes_\SS \mathcal{R} \cong \mathcal{R} \otimes_{\SS} \FF_2 \cong \bigoplus_{K} M_K \]
is an isomorphism of $\Lambda(\Q_1, \Pt)$-modules. Consequently, we have the following lemma. 
\begin{lem} \label{lem:keyiso}Let $\SS_K \subset \SS$ denote the subalgebra $\Lambda(t_Ix_I| I \subset \mathbb{N}_{+}-K)$.  There is a $\Lambda(\Q_1, \Pt)$-module isomorphism 
\[  \bigoplus \limits_{K \underset{\text{finite}}{\subset}\mathbb{N}_{+}} \SS_K \otimes M_K  \cong \mathcal{R} .\]
\end{lem}
\begin{proof} Consider the $\FF_2$-vector space isomorphism 
\[ \iota: \mathcal{R} \to \bigoplus \limits_{K \underset{\text{finite}}{\subset}\mathbb{N}_{+}} \SS_K \otimes M_K \]
which sends \[ t_I x_J \mapsto t_{I \cap J} x_{I \cap J} \otimes t_{I - (I \cap J)} x_{J - (I \cap J)}  \in \SS_K \otimes M_K,\] where $K = I \cup J - I \cap J$. The map $\iota^{-1}$ sends 
\[ t_Bx_B \otimes t_{K -A} x_{A} \mapsto t_{B \cup (K -A) }\cdot x_{B \cup A},\]
where  $A \subset K$.  This map is also a $\Lambda(\Q_1, \Pt)$-module isomorphism as $\SS_K \subset \SS$ is a trivial $\Lambda(\Q_1, \Pt)$-module  by Lemma~\ref{lem:S}. 
 %Keeping in mind that $\SS$ is a trivial $\Lambda(\Q_1, \Pt )$-module (Lemma~\ref{lem:S}) and  \eqref{eq:total-W},   we have isomorphisms of $\Lambda(\Q_1, \Pt )$-modules
%\begin{eqnarray*}
%\SS \otimes( \bigoplus \limits_{K \underset{\text{finite}}{\subset}\mathbb{N}_{+}} M_K) &\iso& \SS \otimes( \bigoplus \limits_{K \underset{\text{finite}}{\subset}\mathbb{N}_{+}} \Ft\langle \W_K\rangle )\\
%&\iso& \SS \otimes \Ft\langle \W \rangle \\
%&\iso& \SS \otimes (\Ft \otimes_{\SS} \R) \\
%&\iso& (\SS \otimes \Ft )\otimes_{\SS} \R \\
%&\iso & \SS \otimes _{\SS} \R \\
%&\iso& \R.
%\end{eqnarray*}
\end{proof}
Hence we can reduce the problem of computing $\MM(\R, \Pt)$ to computing $\MM(M_K, \Pt)$ for various finite subsets $K$ of $\mathbb{N}_+$. Thus, we first need to understand the structure of $M_K$ as a $\Lambda(\Q_1, \Pt)$-module. 
\begin{figure}
\centering
\begin{tikzpicture}[scale = .6]
\node (a0) at (0,0) {$t_1 t_2$};
\node (a1) at (0,2) {$t_1x_2$};
\node (a2) at (4,2) {$t_2x_1 + t_1x_2$};
\node (a3) at (4,4) {$x_1 x_2$};
\draw[blue, bend left] (a0) to (a1);
\draw[blue, bend right] (a2) to (a3);
\draw[red] (a0) -- (2,0) -- (2,4)-- (a3);
\end{tikzpicture}
\caption{$M_{[2]}$ as a module over $\Lambda(\Q_1, \Pt)$, where $[2] = \{ 1, 2 \}$ }
 \label{pictensor2}
\end{figure} 
\begin{figure}
\centering
\begin{tikzpicture}[scale = .7]
\node (a0) at (0,0) {$t_1t_2t_3$};
\node (a1) at (-2,2) {$t_1t_2x_3 + t_1t_3x_2 + t_2t_3x_2$};
\node (a2) at (2,4) {$t_3x_1x_2 + t_2x_1x_3 + t_1x_2x_3$};
\node (a3) at (0,6) {$x_1x_2x_3$};
\node (a4) at (7,2) {$t_1t_2x_3 + t_2t_3x_1$};
\node (a5) at (7,4) {$t_3x_1x_2 + t_1x_2x_3$};
\node (a6) at (11,2) {$t_1t_3x_2 + t_1t_2x_3$};
\node (a7) at (11,4) {$t_2x_1x_3 + t_3x_1x_2$};
\draw[blue, bend left] (a4) to (a5);
\draw[blue, bend left] (a6) to (a7);
\draw[blue, bend left] (a0) to (a1);
\draw[blue, bend right] (a2) to (a3);
\draw[red] (a0) -- (2,0) --  (a2);
\draw[red] (a1) -- (-2,6) -- (a3);
\end{tikzpicture}
\caption{$M_{[3]}$ as a module over $\Lambda(\Q_1, \Pt)$,  where $[3] = \{ 1, 2, 3 \}$}
\label{pictensor3}
\end{figure}

\begin{rem}\label{rem:MJisMn}Let  $[n]$ denote the indexing set $\{1, ..., n\} \subset \mathbb{N}_+$. 
If $|K| = n$, then there exists the unique order preserving bijection
 \[\iota: [n] \to K\]
 and it
induces  an isomorphism $\iota: M_{[n]}\xra{\cong} M_K$. Thus it is enough to understand $\Lambda(\Q_1, \Pt)$-module structure of $M_{[n]}$ for all $n \in \mathbb{N}_+$.
\end{rem}

As depicted in Figure~\ref{pictensor3}, $M_{[3]}$ splits as a   $\Lambda(\Q_1, \Pt)$-module 
\begin{equation}
\label{eqn:M3}
\begin{array}{cccc}M_{[3]} \iso&&  \Lambda(\Q_1, \Pt)\set{x_1x_2x_3} \\ & \oplus& \Lambda(\Q_1)\set{t_3x_1x_2 + t_1x_2x_3}  \\& \oplus& \Lambda(\Q_1)\set{t_2x_1x_3 + t_3x_1x_2}
\end{array}
\end{equation}
as  summand of a free  $\Lambda(\Q_1, \Pt)$-module  and two copies of $\Lambda(\Q_1)$.
\begin{rem}
The splitting of \eqref{eqn:M3} is a consequence of Lemma \ref{lem:P21Q1=Q1P21}.
Since $\Lambda( \Q_1, \Pt)$ is cocommutative, for any $\Lambda( \Q_1, \Pt)$-module $M$ and $\sigma \in \FF_2[\Sigma_n]$, the induced map 
\[ \sigma : M^{\otimes n} \to M^{\otimes n}\]
is a map of $\Lambda( \Q_1, \Pt)$-modules. Note that in the group ring $\FF_2[\Sigma_3]$, the identity element can be written as a sum of idempotent elements 
\[ \mathbbm{1} = e + f_1 + f_2. \]
For example, one can choose $e = \mathbbm{1} + (1 \ 2 \ 3 ) + (1 \ 2 \ 3)$, $f_1= \mathbbm{1} + (1 \ 2) + (1 \ 3 ) + (1 \ 3 \ 2)$ and $f_2 = \mathbbm{1} + (1 \ 2) + (1 \ 3 ) + (1 \ 2 \ 3)$.  Then we have
\[ M^{\otimes 3} \cong e(M^{\otimes 3}) \oplus f_1(M^{\otimes 3}) \oplus f_2(M^{\otimes 3}). \] 
When  $M\cong\Lambda(\Q_1)$, we get the decomposition of of \eqref{eqn:M3}. 
\end{rem}
%\begin{equation}
%\label{eqn:M3}
%M_{[3]} \iso \Lambda(\Q_1, \Pt)\set{t_1t_2t_3}\oplus\Lambda(\Q_1)\set{t_1t_2x_3 + t_2t_3x_1} \oplus \Lambda(\Q_1)\set{t_1t_3x_2 + t_1t_2x_3}.
%\end{equation}
The splitting of \eqref{eqn:M3},  along with the following fact about finite dimensional Hopf algebras, is the key to understanding the structure of  $M_K$.
\begin{thm}[{\cite{NicholsZoeller}}] \label{thm:finitehopf} If $\mathcal{H}$ is a finite dimensional connected Hopf algebra over a field $\FF$, then for any $\mathcal{H}$-module $M$, $\mathcal{H} \otimes M$ is a free $\mathcal{H}$-module. 
\end{thm}

%\begin{rem}
%There are multiple ways to choose a basis for $M_J$.
% The basis of $M_{[3]}$ we chose in Figure~\ref{pictensor3} is $\mathcal{B}_{[3]}$ in Definition~\ref{defn:name} in Section~\ref{basis}. 
%However, for the purposes of this section, this choice will not play any role other than the fact that it is convenient to see the isomorphism of \eqref{eqn:M3} in Figure \ref{pictensor3}.
%\end{rem}

%\begin{defn} Let $\mathcal{E}$ be any $\Ft$-algebra. Let $M$ and $N$ be $\mathcal{E}$-modules. We say that $M$ is stably isomorphic to $N$ if 
%\[
%F \oplus M \iso F' \oplus N 
%\]
%where $F$ and $F'$ are free $\mathcal{E}$-modules. 
%\end{defn}
Let us denote by $A$ the $\Lambda(\Q_1, \Pt)$-module isomorphic to $\Lambda(\Q_1)$ and let $B:= A \otimes A$.
% and denote by $A[n]$ the module obtained by shifting the internal degrees of $A$ by $n$.

Then using \eqref{eqn:M3} and  Theorem~\ref{thm:finitehopf}, we notice  that
\begin{gather*}
M_{[3]} \cong B \otimes A \cong \{\text{Free}\} \oplus A^{\oplus 2}\\
M_{[4]} \cong \{\text{Free}\} \oplus B^{\oplus 2}\\
M_{[5]} \cong \{\text{Free}\} \oplus A^{\oplus 4}
\end{gather*}
where $\set{\text{Free}}$ denotes a free $\Lambda(\Q_1, \Pt)$ module.
%If we want to specify the generators, we can write, for example, $A\{t_3\}=M_3$ or $B\{t_1t_2\}=M_{[2]}$.
 % Note that, as $\Lambda(\Q_1, \Pt)$ modules we have the isomorphisms
  %\[B\cong A\otimes A\]
   %and \[M_{[3]}\cong B\otimes A\cong \{\text{Free}\}\oplus A \oplus A.\]
%When we compute $M_{[4]}$ using Theorem \ref{thm:finitehopf}, we have 
%\[
%M_{[4]} \cong \{\text{Free}\} \oplus B \oplus B.
%\]
This iterative process can be continued as described in Lemma \ref{lem:Mn} below. We use $A\set{y}$, resp. $B\set{y}$, to specify that $y$ generates $A$, resp. $B$, as a $\Lambda(\Q_1, \Pt)$ module. For example, $M_i\cong A\set{x_i}$.
\begin{lem}\label{lem:Mn}
As a $\Lambda(\Q_1, \Pt)$-module,
\[
M_{[2r+1]} \cong \{\text{Free}\}\oplus (\bigoplus_{i=1}^{2^r} A\{h_{2r+1,i}\}),
\]
 where $\ell(h_{2r+1,i})=r+1$.

As a $\Lambda(\Q_1, \Pt)$-module,
\[M_{[2r]}\cong \{\text{Free}\} \oplus \bigoplus_{i=1}^{2^{r-1}}  B\{h_{2r,i}\},
\] 
 where $\ell(h_{2r,i})=r+1$. 
\end{lem}
\begin{proof}
Our proof is by induction on $r.$
From Figure~\ref{picbasic}, Figure~\ref{pictensor2} and Figure~\ref{pictensor3}, the claim is true for $k= 1, 2,3$. Note that 
\[
h_{1,1}=x_1 \qquad \quad h_{2,1}=x_1x_2 \qquad  \quad h_{3,1}=(t_3x_1+x_3t_1)x_2 \qquad h_{3,2}=(t_2x_3+t_3x_2)x_1
\]
Now assume that the result is true for $2r-1$, i.e. 
\[ M_{[2r-1]} \iso \{\text{Free}\} \oplus \bigoplus_{1 \leq  i  \leq 2^{r-1}} A\{h_{2r-1,i}\}\]
where $\ell(h_{2r-1,i})=r$   and $\{\text{Free}\}$ is a free $\Lambda(\Q_1, \Pt)$-module. It follows that 
 \[ M_{[2r]} \iso 
M_{[2r-1]} \otimes M_{2r} \iso 
 (\{\text{Free}\} \otimes A\{x_{2r}\}) \oplus \bigoplus_{1 \leq  i  \leq 2^{r-1}} B\{h_{2r-1, i} \cdot  x_{2r}\}. \] 
 By Theorem~\ref{thm:finitehopf}, the first summand is, again, a free module. Set 
 \[h_{2r, i} = h_{2r-1, i} \cdot  x_{2r}\]
  and notice 
 $\ell(h_{2r-1,i} x_{2r}) = \ell(h_{2r-1,i}) + \ell(x_{2r}) = r+1. $
 
 To complete the inductive argument, observe 
 \begin{eqnarray*}
    M_{[2r+1]}  &\iso& M_{[2r-1]} \otimes B\{x_{2r} x_{2r+1}\}  \\ 
   &\iso & (\{\text{Free}\} \oplus \bigoplus\limits_{1 \leq  i  \leq 2^{r-1}} A\{h_{2r-1,i}\}) \otimes B\{x_{2r} x_{2r+1}\} \\
   &\iso &\{\text{Free}\} \oplus  \bigoplus\limits_{1 \leq  i  \leq 2^{r-1}} (A\{h_{2r+1, 2i-1}\} \oplus A \{ h_{2r+1, 2i} \}), 
 \end{eqnarray*}
 where  one can define the generators $h_{2r-1,j}$ from  Figure~\ref{pictensor3}  by replacing $x_1, x_2, x_3$ with $h_{2r-1,i}$, $x_{2r}$ and $x_{2r+1}$ respectively.  More specifically, one can define 
 \[ h_{2r+1, 2i-1} =  \Q_1(h_{2r-1,i} \cdot x_{2r+1} ) \cdot x_{2r} \]
 \[ h_{2r+1, 2i} = h_{2r-1,i} \cdot  \Q_1( x_{2r} x_{2r+1}) . \]
 It is easy to check that $\ell(h_{2r+1, j}) = r+1 $.
 %The result follows from the $\Lambda(\Q_1, \Pt)$ module structure of $M_{[3]}$ (see Figure~\ref{pictensor3}). 
\end{proof}

%Note that, $M_J \iso M_{[n]}$, where $|J|=n$. Therefore we can conclude that:
%\begin{cor}\label{cor:MJ}
%Let $J=\set{i_1, \ldots, i_n}\subset \mathbb{N}.$ If $n=2t+1$ then $M_J$ is  isomorphic to a direct sum of a free module  with  $2^t$ copies of $A$. If $n=2t,$ then $M_J$ is  isomorphic to a direct sum of a free module with $2^{t-1}$ copies of $B$.
%\end{cor} \prasit{This is repeating the previous result, but in words.  }
Following the proof of Lemma~\ref{lem:Mn},  we can provide an explicit basis of $\MM(M_K, \Pt)$.  By Remark \ref{rem:MJisMn} it suffices to provide a basis for $\MM(M_{[n]}, \Pt)$ for all $n \geq 1$. We do so inductively (see Definition~\ref{defn:name}), however we must treat the odd and the even case separately, essentially because of Lemma~\ref{lem:Mn}. Since $A$ is a trivial $\Lambda(\Pt)$-module, $\MM(A, \Pt)\cong A$, and we get 
\[ \MM(M_{[2r+1]}, \Pt) \cong \MM\left(\bigoplus_{i=1}^{2^r} A\{h_{2r+1,i}\}, \Pt\right) \cong \bigoplus_{i=1}^{2^r} A\{h_{2r+1,i}\}. \]
 Thus the collection 
 \[ \{ h_{2r+1,i}: 1 \leq i \leq 2^r \} \cup \{ \Q_1( h_{2r+1,i}): 1 \leq i \leq 2^r \}\] 
  is an $\FF_2$-basis of $\MM(M_{[2r+1]}, \Pt).$ When $n$ is even, say $n =2r$, then  
 \[ \MM(M_{[2r]}, \Pt) \cong  \bigoplus_{ i = 1 }^{2^{r-1} }\MM( B \{ h_{2r, i }\}, \Pt ).\]
 Now note that, if $B\{ x\otimes y\} = A \{ x\} \otimes A\{ y\}$ (where $x$ and $y$ are generators), then \[ \{ \Q_1(x) \otimes y , x \otimes \Q_1 (y)\} \]  is an $\FF_2$-basis of $\MM(B \{ x \otimes y\}, \Pt)$. Using the fact that 
 \[ h_{2r, i} = h_{2r-1, i} \cdot x_{2r},\] 
 we get   Corollary~\ref{cor:dimension} and Definition~\ref{defn:name} thereafter.

\begin{cor} \label{cor:dimension} Let $\MM(M_K, \Pt)_l = \set{x \in \MM(M_K, \Pt)| \ell(x) = l}$. 

If  $|K| = 2r+1$, then  
 \[\dim  \MM(M_K, \Pt)_l = \left\lbrace \begin{array}{cccc}
2^{r}, & \text{if $l=r, r+1$} \\
0, & \text{otherwise.}
 \end{array}\right. \]
 If $|K| = 2r$, then 
\[\dim  \MM(M_K, \Pt)_l = \left\lbrace \begin{array}{cccc}
2^r, & \text{if $l=r$} \\
0, & \text{otherwise.}
 \end{array}\right. \]

\end{cor}
\begin{proof}
Lemma~\ref{lem:Mn} implies 
\[ 
M_{[2r+1]} \iso \{\text{Free}\} \oplus \bigoplus_{1 \leq i \leq 2^{r}} A\set{h_{2r+1,i}} 
\]
where $\ell(h_{2r +1, i}) = r+1.$ By Lemma \ref{Lem:action-reduces-length}  we have $\ell(\Q_1(h_{2r +1, i})) = r$. Thus $\{ h_{2r+1, i} \}$ is the  basis for  $\MM(M_{[2r+1]}, \Pt)_{r+1}$ and  
$\{ \Q_1(h_{2r+1, i})\} $ is the basis for $\MM(M_{[2r+1]}, \Pt)_{r}$. 
Applying Remark \ref{rem:MJisMn} we deduce the statement about dimension for any $M_K$ with $|K|=2r+1$.

For the even case we have from Lemma~\ref{lem:Mn} 
\[ M_{[2r]} \iso \{\text{Free}\} \oplus \bigoplus_{1 \leq i \leq 2^{r-1}} B\set{h_{2r, i}} \]
where $\ell(h_{2r,i}) = r+1$. Then for each $i$, $\MM(B\set{h_{2r, i}}, \Pt) = \MM(B\set{h_{2r, i}}, \Pt)_{r} $ is an $\FF_2$ vector space of dimension $2$ generated by $\{ h_{2r-1, i } \cdot t_{2r}, \Q_1(h_{2r-1,i} )\cdot x_{2r}\}$ .
\end{proof}

\begin{defn} \label{defn:name} We define the basis $\mathcal{B}_{[n], l}$ of $\MM(M_{[n]}, \Pt)_l$ for $0 \leq l \leq n$ inductively starting with $\mathcal{B}_{[1], 0}= \{t_1\}$ and $\mathcal{B}_{[1],1}= \{x_1 \}$. Suppose we have defined 
\[ 
 \mathcal{B}_{[2r-1], l} := \left\lbrace
\begin{array}{cll} 
  \lbrace h_{2r-1, 1}, \dots, h_{2r-1, 2^{r-1}} \rbrace &\text{ if $l = r$} \\
    \lbrace \Q_1(h_{2r-1,1}), \dots, \Q_1(h_{2r-1, 2^{r-1}}) \rbrace &\text{ if $l = r-1$} \\
    \emptyset & \text{ otherwise.}
 \end{array} \right.
 \]
 Then define:
 \[ 
 \mathcal{B}_{[2r], r} := 
 \{  h_{2r-1, 1} \cdot t_{2r}, \dots, h_{2r-1, 2^{r-2}} \cdot t_{2r}  \} \cup  \{ \Q_1( h_{2r-1, 1} )\cdot x_{2r}, \dots, \Q_1( h_{2r-1, 2^{r-2}}) \cdot x_{2r}  \}  \]
and set $\mathcal{B}_{[2r], l} := \emptyset$ if $l \neq r$.  

Now define $h_{2r+1, 2i-1} = \Q_1(h_{2r-1, i} ) \cdot ( x_{2r+1} \cdot x_{2r})$ and  $h_{2r+1, 2i} = h_{2r-1, i} \cdot \Q_1(x_{2r}x_{2r+1})$ and set 
\[ 
 \mathcal{B}_{[2r+1], l} := \left\lbrace  \begin{array}{cll}
  \lbrace h_{2r+1, 1}, \dots, h_{2r+1, 2^{r-2}} \rbrace &\text{ if $l = r+1$} \\
    \lbrace \Q_1(h_{2r+1,1}), \dots, \Q_1(h_{2r+1, 2^{r-2}}) \rbrace &\text{ if $l = r$} \\
    \emptyset & \text{ otherwise.}
 \end{array} \right.
 \]
We let $\mathcal{B}_{[n]}$ denote the union $\bigcup_{l}  \mathcal{B}_{[n], l}$.  Let $\mathcal{B}_K$ denote the image of the $\mathcal{B}_{[n]}$ under the isomorphism  
$
\iota: M_{[n]} \to M_{K}
$
of Remark \ref{rem:MJisMn}.
\end{defn}

\begin{ex}[Examples of $\mathcal{B}_{K}$] \label{ex:sample}
We explicitly identify $\mathcal{B}_{[n]}$ using Definition~\ref{defn:name} for $n \leq 4$, and for $n=1,2,3$ we can compare to Figures \ref{picbasic}, \ref{pictensor2} and \ref{pictensor3}, to see that $\mathcal{B}_{[n]}$ is indeed the basis for $\MM(M_{[n]}, \Pt)$.
\begin{itemize}
\item  $\mathcal{B}_{[1]} = \set{t_1, x_1}$,
\item  $\mathcal{B}_{[2]} = \set{t_1x_2, t_2x_1}$,
\item $\mathcal{B}_{[3]} = \set{t_1t_2x_3 + t_1t_3x_2, t_1t_2 x_3 + t_2t_3x_1 }\cup \set{ t_3 x_1x_2 + t_2x_1x_3, t_3x_1x_2  + t_1x_2x_3 }$, and,
\item $\mathcal{B}_{[4]} = \set{t_1t_2x_3x_4 + t_1t_3x_2x_4, t_1t_2 x_3x_4 + t_2t_3x_1x_4, t_3t_4 x_1x_2 + t_2t_4x_1x_3, t_3t_4x_1x_2  + t_1t_4x_2x_3}.$
\end{itemize}
\end{ex}
Let $P_K := \FF_2 \langle \mathcal{B}_{K} \rangle \subset M_K$. By construction, $P_K$ is a trivial split $\Lambda(\Pt)$-submodule of $M_K$. The splitting of $P_K$ follows from the fact that its complement is free and the fact that $\Lambda(\Pt)$, like any other finite dimensional connected Hopf algebra, is self-injective \cite[Ch. 12, \S 2]{Margolis_book}.  

\iffalse
Now we are ready to prove our main theorem. The main idea is as follows: consider an element $b\in \mathcal{B}_{J}$ for some $J$ which is an element of $\MM(\tmf, \Pt)$. Since the action of $\Pt$ is linear with respect to the elements from $\mathcal{S}=\set{x_It_I|I \subset \mathbb{N}_{+}}$, we can define a product $x_It_I\otimes b \in S\otimes \MM(\tmf, \Pt)$. There are two possibilities:
\begin{enumerate}
\item If $I\cap J\neq \emptyset$, then  $x_It_I\otimes b=0$, since each $x_i^2=t_i^2=0$;
\item If $I\cap J= \emptyset$, then $x_It_I\otimes b$ is a nonzero element in $\MM(\tmf, \Pt)$ because of Lemma~\ref{d2linear}.
\end{enumerate}
As a simple example, consider the element 
 \[ x_1t_2 \in  \MM(M_{[2]}, \Pt) \subset \MM(\tmf, \Pt).\] Then the elements such as  $ x_3t_3 \cdot x_1 t_2 $ and  $ x_4t_4 \cdot x_1 t_2$ are nonzero elements in  $\MM(\tmf, \Pt)$, whereas \[ x_1t_1 \cdot x_1t_2 =x_1^2t_1t_2=0.\]
This observation leads to the next theorem.
\fi

\begin{thm}\label{thm:basistmf} Let $K$ be a finite subset of $\mathbb{N}_{+}$. 
%Let$\mathcal{S}_K = \Lambda(t_ix_i|i \in \mathbb{N}_{+} -K)$, and let
Let \[ \mathcal{SB}_K:=  \lbrace t_{I}x_{I} \cdot b\,|\, I \cap K = \emptyset \text{ and } b \in \mathcal{B}_K  \rbrace \subset \mathcal{R}. \] 
Then 
\[
\mathcal{B} := \bigsqcup\limits_{K \underset{\text{finite}}{\subset}\mathbb{N}_{+}} \mathcal{SB}_K
\]
forms a basis of the $\Ft$-vector space $\MM(\tmf, \Pt)$ and
\[ \MM(\tmf, \Pt)   \iso 
 \bigoplus \limits_{K \underset{\text{finite}}{\subset}\mathbb{N}_{+}}  \Ft \langle \SS\mathcal{B}_K \rangle
\iso \bigoplus \limits_{K \underset{\text{finite}}{\subset}\mathbb{N}_{+}} \SS_K \otimes  \MM(M_K, \Pt) .
\]
is an isomorphism of $\Ft$-vector spaces.
\end{thm}
\begin{proof} 
By Lemma \ref{lem:keyiso}, we have a $\Lambda(\Q_1, \Pt)$ module isomorphism 
\[
\mathcal{R} \cong \bigoplus_{K \underset{\text{finite}}{\subset} \mathbb{N}_+}  \SS_K \otimes M_K .
\]
Therefore, the linearity of the action of $\Pt$ (see Corollary~\ref{cor:S}) with respect to elements in $\SS$ gives us 
\begin{eqnarray*}
\MM(\tmf, \Pt) &\iso& \MM(\R, \Pt) \\ 
&\iso& \MM( \bigoplus_{K \underset{\text{finite}}{\subset} \mathbb{N}_+} \SS_K \otimes M_K, \Pt) \\
&\iso&  \bigoplus_{K \underset{\text{finite}}{\subset} \mathbb{N}_+} \SS_K \otimes \MM( M_K, \Pt) \\
&\iso&  \bigoplus_{K \underset{\text{finite}}{\subset} \mathbb{N}_+} \SS_K \otimes P_K \\
&\iso& \bigoplus_{K \underset{\text{finite}}{\subset} \mathbb{N}_+} \FF_2 \langle \SS\mathcal{B}_K \rangle.
\qedhere
\end{eqnarray*}
%The last isomorphism above follows from the observation that  for  $s \in \SS$ and $w \in \Ft\langle \W_J \rangle$ $s \otimes w \neq 0$ if and only if  $s \in \SS_J$. Consequently  an $\mathbb{F}_2$ vector space $\MM(\tmf, \Pt)$   has basis $\mathcal{B}$, and as an $\mathcal{S}$ module $\MM(\tmf, \Pt)$ is generated by $\bigsqcup\limits_{J \underset{\text{finite}}{\subset}\mathbb{N}_{+}} \mathcal{B}_J$.
\end{proof}
\begin{rem} Let $e$ denote the exchange map $e: \mathcal{R} \to \mathcal{R} $
which sends \[e: t_I x_J \mapsto t_J x_I.\] It seems to be the case that 
%$0 \neq m  \in \R$  represents a class in $\MM(\tmf, \Pt)$ (i.e. $m$ is a $\Pt$-cycle and not a $\Pt$-boundary) so does $e(m)$. 
$[m] \in \MM(\tmf, \Pt)$ if and only if $[e(m)] \in \MM(\tmf, \Pt)$. 
The source of such symmetry is unclear to the authors, although it might be related to Spanier--Whitehead duality.
\end{rem}
Finally, we would like to say a word about the module structure of $\MM(\tmf, \Pt)$ over $\mathcal{S}$. Note that the collection of elements 
\[ \mathcal{B}_{\mathcal{S}} := \{ t_Ix_I| I  \underset{\text{finite}}\subset \mathbb{N}_+ \}\] 
forms an $\FF_2$-basis of $\mathcal{S}$. The $\mathcal{S}$-module structure on $\MM(\tmf, \Pt)$ is extended from a pairing at the level of bases
\begin{align*}
\mathcal{B}_{\mathcal{S}} \otimes \mathcal{S} \mathcal{B}_K &\xra{\mu} \mathcal{S}\mathcal{B}_K \\
s \otimes( s' \cdot  b)&\mapsto 
\begin{cases}
(s \cdot s' ) \cdot b, \text{ if } I \cap  K = \emptyset  \\
0, \text{ if } I \cap K \neq \emptyset.
\end{cases} 
\end{align*}

\begin{rem} Recall that $H_*(\tmf)$ was described in terms of $\zeta_i$. We can convert an element of the Margolis homology expressed in terms of $t_i$ and $x_i$  back to an expression involving $\zeta_i$ using the identifications of Notation~\ref{Notn:xt}. For example, \[ t_4t_9x_2x_6 + t_2t_9x_4x_6 \]
can be identified with the  class represented by element $\zeta_5^5\zeta_{10}^4 \zeta_9 + \zeta_3^4 \zeta_{10}^4 \zeta_7 \zeta_9 \in \mathfrak{T}$.
\end{rem}

\section{\texorpdfstring{$P_2^1$}{P21} Margolis homology of $\tmf^{\sma r}$ and $B(\mathbb{Z}/2^{\times n})_+$} \label{sec:misc}
\subsection{$\Pt$ Margolis homology of  $\tmf^{\sma r}$}
Note that $$H_*(tmf^{\wedge r})\iso H_*(tmf)^{\otimes r} \iso  \mathfrak{T}^{\otimes r}.$$ We first extend the notion of length to $\mathfrak{T}^{\otimes r}.$ For a monomial  $\zeta^{I_1}| \dots | \zeta^{I_r}$ for $\zeta^{I_i} \in \mathfrak{T}^{\otimes r}$, which is a tensor product of monomials in $\mathfrak{T}$, we define 
\[ L(\zeta^{I_1}| \dots | \zeta^{I_r}) = L(\zeta^{{I}_1}) + \dots + L(\zeta^{I_r}). \]
We define the even subalgebra $\mathbb{E}_r$  of $\mathfrak{T}^{\otimes r}$ as  the span of those monomials in $\mathfrak{T}^{\otimes r}$ whose lengths are zero. Observe that, 
 $$\mathbb{E}_r \iso \mathcal{E}^{\otimes r}.$$
 Notion of length leads to an increasing filtration on $\mathfrak{T}^{\otimes r}$, call it the length filtration, by setting 
\[
   G^p(\mathfrak{T}^{\otimes r})=\set{ (\zeta^{I_1} | \dots |  \zeta^{I_r}) \,|\, L(\zeta^{I_1}| \dots | \zeta^{I_r} )\leq p}.
\]
 Let $\mathbb{K}_r = \mathcal{K}^{\otimes r}$, where $\mathcal{K}$ is as defined in Section~\ref{Sec:Margolis}.  Just like in the case $r=1$, we get a length spectral sequence and its $E_1$ page is 
 \begin{equation} \label{LSSR}
 E_1^{\bullet} \iso \MM(\mathbb{E}_r, \Pt ) \otimes \mathbb{K}_r \Rightarrow \MM(\tmf^{\sma r}, \Pt). 
 \end{equation}
 Since action of $\Pt$ follows the Leibniz rule when restricted to $\mathcal{E}$, we get  \[ \MM(\mathbb{E}_r, \Pt ) \iso \MM(\mathcal{E},\Pt)^{\otimes r}. \] 

 \begin{notn} \label{notn:xij} For shorthand, we denote $x_{i,j} = (\underbrace{1| \dots |1}_{j-1} |\zeta_{i+3}|\underbrace{1| \dots |1}_{r-j} )$ and $t_{i,j} = (\underbrace{1| \dots |1}_{j-1} |\zeta_{i+1}^4|\underbrace{| \dots |1}_{r-j} )$. With this notation we have $$\Q_1(x_{i,j}) = t_{i,j}.$$ 
 \end{notn}
 Using Notation~\ref{notn:xij}, we see that the $E_1$ page of the length spectral sequence \eqref{LSSR}, as an algebra, is  isomorphic to 
 \[ \mathcal{R}_r :=  \Lambda(t_{i,j}: i \in \mathbb{N}- \set{0}, 1 \leq j \leq r) \otimes \Lambda(x_{i,j}: i \in \mathbb{N}- \set{0}, 1 \leq j \leq r). \]
 It is easy to see that the map induced by the reindexing map $$\iota: (i,j) \mapsto r(i-1) + j,$$ produces a (non-canonical) isomorphism of algebras  between $\mathcal{R}_r$ (the  $E_2$ page  of  \eqref{LSSR}) and $\mathcal{R}$ (the  $E_2$ page of  \eqref{LSS}), after forgetting the internal grading. This is also an isomorphism of $\Lambda(\Q_1, \Pt)$-modules. 
 Thus we have an isomorphism
\[ \iota_*: \MM(\tmf, \Pt) \overset{\iso}\longrightarrow  \MM(\tmf^{\sma r}, \Pt)\]
induced by the $\iota$.
 Therefore, Theorem~\ref{thm:basistmf} essentially gives a complete calculation of $\MM(\tmf^{\sma r}, \Pt)$. 
 %Of course, $\iota_*$ does not preserve internal grading of elements. 
\begin{ex}
For example, let us assume $r =3$. Then the element $t_2t_4x_6x_9 + t_2t_6x_4x_9 \in \MM(\tmf, \Pt)$ (see Example~\ref{ex:sample}) corresponds to the element 
\[ t_{1,2}t_{2,1}x_{2,3}x_{3,3} + t_{1,2}t_{2,3}x_{2,1}x_{3,3} \in \MM(\tmf^{\sma 3}, \Pt)  \]
under the bijection obtained from the above reindexing. 
When expressed in terms of $\zeta_i$s (see Notation~\ref{notn:xij}), the same element can be expressed as 
\[\zeta_3^4| \zeta_2^4| \zeta_5 \zeta_6| 1 + \zeta_5| \zeta_2^4|\zeta_3^4\zeta_6|1.\]
 \end{ex}
 \begin{rem}[$\Pt$ Margolis homology of Brown--Gitler spectra]
 It is well-known that 
 \[ H_*(\tmf) \iso \bigoplus_{i \geq 0} H_*(\Sigma^{8i}bo_i)  \]
 where $bo_i$ are certain Brown--Gitler spectra associated with $bo$. In \cite{Mah81} Mahowald defined a multiplicative weight function, which is given by $w(\zeta_i)= 2^{i-1}$. $H_*(\Sigma^{8i}bo_i)$ is the summand of $H_*(\tmf)$ which consists of elements of Mahowald weight exactly equal to $8i$. 
 %Since, $t_{i,j}$ and $x_{i,j}$  equal  $(1| \dots |1 |\zeta_{i+1}^4|1| \dots |1 )$ and $(1| \dots |1 |\zeta_{i+3}|1| \dots |1 )$ respectively,
 We assign Mahowald weight of  $t_{i,j}$ and $x_{i,j}$ as
 \[ w(t_{i,j}) = w(x_{i,j}) = 2^{i+1}.\]
 It follows that the Margolis homology $\MM(bo_{q_1} \sma \dots \sma bo_{q_r}, \Pt)$ is a summand of $\MM(\tmf^{\sma r}, \Pt)$. It consists of all polynomials of $\MM(\tmf^{\sma r}, \Pt)$ expressed in terms of $x_{i,j}$ and $t_{i,j}$ such that $w(x_{i,j})= w(t_{i,j}) = 4q_j$.  
 \end{rem}
 \begin{rem} While it is true that $\mathcal{R}_r \cong \mathcal{R}^{\otimes r}$, as an $\FF_2$-algebra as well as an $\Lambda(\Q_1, \Pt)$-module, it is not useful for the purposes of calculating   $\MM(\mathcal{R}_r , \Pt)$. This is because $\Pt$ does not obey Leibniz rule and 
 \[ \MM(\mathcal{R}_r , \Pt) \not\cong \MM(\mathcal{R}, \Pt)^{\otimes r}. \]
 However we overcome this difficulty by producing an $\Lambda(\Q_1, \Pt)$-module isomorphism $\iota_*$ at the expense of forgetting the internal grading. 
 \end{rem}
\subsection{ $\Pt$ Margolis homology of  $(B\mathbb{Z}/2^{\times k})_+$}
 The space $B\mathbb{Z}/2$ is also known as $\mathbb{RP}^{\infty}$, the real infinite-dimensional projective space. It is well-known that 
 \[H^*((B\mathbb{Z}/2)_+, \FF_2) \iso \Ft[x]\] and therefore 
 \[ H^*((B\mathbb{Z}/2^{\times k})_+, \FF_2) \iso \Ft[x_1, \dots x_k]. \]
 It can be seen that $\Pt(x_i) = 0$ and $\Q_1(x_i) = x_i^4$. We again define the length function on the monomials in the usual way
 \[ L(x_1^{i_1}\dots x_k^{i_k}) = (i_1 \mod 2) + \dots + (i_k \mod 2).\]
 The even complex $\mathcal{E}$, which is the span of elements of length zero, is isomorphic to \[ \mathcal{E} = \Ft[x_1^2, \dots, x_k^2].  \]
 It can be seen that $\Pt(x_i^2) = x_i^8$. Now observe that $\Q_1$ acts trivially on $\mathcal{E}$, hence $\Pt$ acts as a derivation  and, therefore, 
 \[ \MM(\mathcal{E}, \Pt) \iso \Lambda(x_1^4, \dots, x_k^4).\]
 Now the length function gives us an increasing length filtration
 \[ G^p(\Ft[x_1, \dots, x_k]) =\Ft\langle x_1^{i_1} \dots x_k^{i_k} : L(x_1^{i_1} \dots x_k^{i_k}) \leq p \rangle .\]
 This results in a length spectral sequence which only has $d_0$ and $d_2$ differentials. If we denote  $x_i^4$ by $t_i$ for convenience, we can see that the length spectral sequence 
 \[ E_1^{\bullet }= \Lambda(t_1, \dots, t_k) \otimes \Lambda(x_1, \dots, x_k) \Rightarrow \MM((B\mathbb{Z}/2^{\times k})_+, \Pt)\] 
 is a sub spectral sequence of \eqref{LSS} and is, in fact, isomorphic to it when $k = \infty$. Thus, when $k$ is finite, we can recover a complete description of  $\MM((B\mathbb{Z}/2^{\times k})_+, \Pt)$ from Theorem~\ref{thm:basistmf}. More precisely, we obtain  
 \[ \MM((B\mathbb{Z}/2^{\times k})_+, \Pt) \iso \bigoplus \limits_{K \subset [k]} S_K\otimes  \MM(M_K, \Pt),\]
 where $S_K= \Lambda(t_ix_i\,|\, i \in [k]- K)$. $\MM((B\mathbb{Z}/2^{\times k})_+, \Pt)$ is a module over $\mathcal{S}_{[k]}$.
 \begin{ex} $\MM(\mathbb{RP}^{\infty}_+, \Pt) \cong \FF_2 \langle x_1, t_1, t_1x_1 \rangle$, where the internal degrees of $x_1$ and $t_1$ are $1$ and $4$ respectively and $\mathcal{S}_{[1]} = \Lambda(t_1x_1)$. Similarly, 
 \begin{eqnarray*}
  \MM((\mathbb{RP}^{\infty} \times \mathbb{RP}^{\infty})_+, \Pt)&\cong& \FF_2 \langle x_1,x_2,  t_1, t_2,  t_1x_1, t_2x_2, \\
 && \hspace{.5in} t_1x_2, t_2x_1, t_1x_1x_2, t_2x_2 x_1, t_1t_2x_2, t_1t_2 x_1 \rangle 
  \end{eqnarray*}
 where the internal degrees of $x_i$ and $t_i$ are $1$ and  $4$ respectively. Here $\mathcal{S}_{[2]} = \Lambda(t_1x_1, t_2x_2)$.  If we denote 
 \[ H^*((\mathbb{RP}^{\infty} \times \mathbb{RP}^{\infty})_+) \cong \Ft[y, z]\]
 where $|y|=|z|=1$, then one may choose $x_1 = [ x]$, $x_2 = [y]$, $t_1= [x^4]$ and  $t_2 = [y^4]$. 
  \end{ex}
 \bibliographystyle{alpha}
 \bibliography{P21.bib}
\end{document}